\documentclass[10pt]{article} 
\usepackage[letterpaper, margin=1in]{geometry} 
\usepackage[english]{babel}
\usepackage[normalem]{ulem}
\usepackage{amssymb,amsmath,amsthm,mathtools,braket,aligned-overset}
\usepackage{thmtools,thm-restate}
\usepackage{graphicx,xcolor}
\usepackage{float}
\usepackage{caption,subcaption}
\usepackage{tikz}
\usetikzlibrary{calc,decorations.pathmorphing,decorations.text,decorations.markings,matrix}

\usepackage{array,booktabs}
\usepackage{enumerate}

\usepackage{appendix}
\usepackage{titlesec}
\usepackage{authblk}

\usepackage[hidelinks]{hyperref}
\usepackage[capitalise,nameinlink,noabbrev]{cleveref}
\usepackage[noadjust]{cite}

\theoremstyle{plain}
\newtheorem{thm}{Thm}[section]

\newtheorem{claim}{Claim}
\numberwithin{claim}{section}
\newtheorem{theorem}[thm]{Theorem}

\newtheorem{proposition}[thm]{Proposition}

\newtheorem{problem}{Problem}

\newcommand{\e}{\varepsilon}
\newcommand{\D}{\mathcal{D}}

\allowdisplaybreaks[4]

\begin{document}

\title{Bounded chromatic number of graphs with small clique number and large minimum degree}

 \date{}
\author[ ]{Jiaao Li}
\author[ ]{Xinyuan Li}

\affil[ ]{\small School of Mathematical Sciences and LPMC, Nankai University, Tianjin 300071, China}
\affil[ ]{\small Emails: lijiaao@nankai.edu.cn; xinyuanli@mail.nankai.edu.cn}

\maketitle

\begin{abstract}
We prove that every triangle-free graph with minimum degree at least $\frac{n}{3}$ is $4$-colorable and thereby settle a problem of Brandt and Thomass\'e (2005) at the threshold $\frac{n}{3}$. The number four is best possible. For a positive integer-valued function $f(n)=o(n)$, we relate the chromatic number of $f(n)$-vertex subgraphs of the Kneser graph $KG(n,f(n))$ to that of triangle-free graphs with minimum degree at least $\frac{n}{3}-f(n)$. Consequently, for every $0<\delta<1$ and $\e>0$, and for all sufficiently large $n$, every $n$-vertex triangle-free graph with minimum degree at least $\frac{n}{3}-n^{1-\delta}$ has chromatic number at most $10^{391}+1+\left\lceil{(1+\e)(1-\delta)}/{\delta}\right\rceil$.

We also show that every sufficiently large $n$-vertex maximal triangle-free graph with minimum degree at least $\frac{n}{3}-f(n)$ and chromatic number at least $10^{391}$ contains a bipartite subgraph with parts of orders $\frac{n}{3}-O(f(n))$ and $\frac{2n}{3}-O(f(n))$; the remaining induced subgraph admits a homomorphism to $KG(\frac{n}{3}-O(f(n)),O(f(n)))$. Finally, we connect maximal $K_r$-free graphs with minimum degree at least $\frac{2r-5}{2r-3}n-f(n)$ to $K_{r-1}$-free graphs and extend these results to $K_r$-free graphs.  Our proofs employ the recent strong Brandt--Thomass\'{e}  theorem of \L uczak, Polcyn, and Reiher.
\end{abstract}

\section{Introduction}\label{sec:introduction}
The problem of bounding the chromatic number of a graph with small clique number and large minimum degree has received much attention since it was proposed by Erd\H{o}s and Simonovits \cite{1973-erdos}. In 1974, Andr\'{a}sfai, Erd\H{o}s, and S\'{o}s \cite{1974-andrasfai} proved that every $n$-vertex $K_r$-free graph with minimum degree greater than $\frac{3r-7}{3r-4}n$ is $(r-1)$-colorable. Later, Jin \cite{1993-jin} proved that every triangle-free graph with minimum degree greater than $\frac{10}{29}n$ is homomorphic to an Andr\'{a}sfai graph. The Andr\'{a}sfai graph $\Gamma_i$ is the Cayley graph with vertex set $\{0,\dots,3i-2\}$ in which two vertices are adjacent if and only if their difference modulo $3i-1$ belongs to $\{i,\dots,2i-1\}$ (see Figures~\ref{fig:andrasfai graph1} and~\ref{fig:andrasfai graph2} for two drawings of $\Gamma_3$). Jin \cite{1995-jin} also determined a minimum-degree condition guaranteeing that a triangle-free graph is $3$-colorable. Brandt and Pisanski \cite{1998-brandt} then found another infinite family of $4$-chromatic triangle-free graphs with minimum degree greater than $\frac{n}{3}$, namely the blow-ups of Vega graphs. As shown in Figure~\ref{fig:vega graph}, a Vega graph $\Upsilon_{i}^{00}$ has three parts. The first induces an Andr\'{a}sfai graph $\Gamma_i$ with $V(\Gamma_i)=\{0,\dots,3i-2\}$. The second induces the $6$-cycle $awbucva$. The third is an edge $xy$. The vertex $x$ is adjacent to $a,b,c$, and $y$ is adjacent to $u,v,w$. The vertices $a$ and $u$ are adjacent to $\{0,\dots,i-1\}$, the vertices $b$ and $v$ to $\{i,\dots,2i-1\}$, and the vertices $c$ and $w$ to $\{2i,\dots,3i-2\}$. The other three Vega graphs are $\Upsilon_{i}^{10}=\Upsilon_{i}^{00}-\{y\}$, $\Upsilon_{i}^{01}=\Upsilon_{i}^{00}-\{2i-1\}$, and $\Upsilon_{i}^{11}=\Upsilon_{i}^{00}-\{y,2i-1\}$. In particular, the Gr\"otzsch graph $\Upsilon:=\Upsilon_{2}^{11}$ is shown in Figure~\ref{fig:grotzsch graph}.
In 2002, Thomassen \cite{2002-thomassen} showed that every triangle-free graph with minimum degree at least $(\frac{1}{3}+\varepsilon)n$ has bounded chromatic number for each $\varepsilon>0$, thereby settling the Erd\H{o}s--Simonovits problem.
Later, Brandt and Thomass\'{e} \cite{2010-brandt} proved that every maximal triangle-free graph with minimum degree strictly greater than $\frac{n}{3}$ is a blow-up of a Vega graph and hence is $4$-colorable.
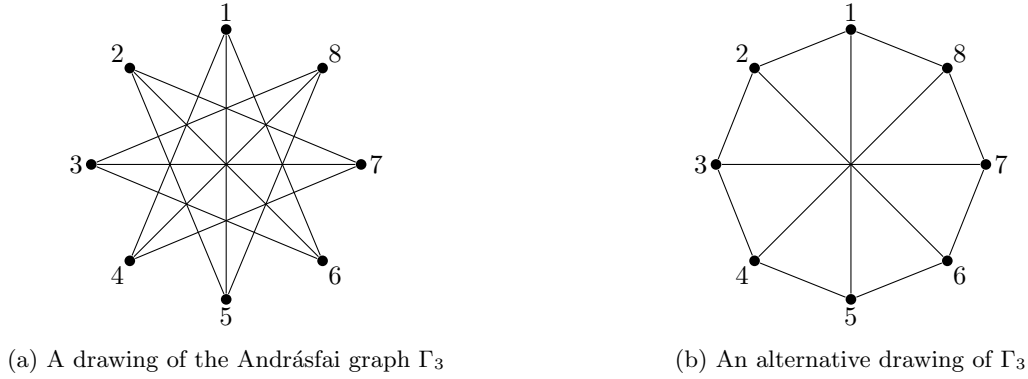
\begin{figure}
    \centering
\begin{subfigure}[b]{0.45\textwidth}
        \centering
        \begin{tikzpicture}[scale=0.85]
        \node [circle, fill=black, inner sep=0pt, minimum size=4pt] (v1) at (-0.5,1.6) {};
\node [circle, fill=black, inner sep=0pt, minimum size=4pt] (v8) at (-2,1) {};
\node [circle, fill=black, inner sep=0pt, minimum size=4pt] (v6) at (-2.6,-0.5) {};
\node [circle, fill=black, inner sep=0pt, minimum size=4pt] (v2) at (-2,-2) {};
\node [circle, fill=black, inner sep=0pt, minimum size=4pt] (v3) at (-0.5,-2.6) {};
\node [circle, fill=black, inner sep=0pt, minimum size=4pt] (v4) at (1,-2) {};
\node [circle, fill=black, inner sep=0pt, minimum size=4pt] (v7) at (1.6,-0.5) {};
\node [circle, fill=black, inner sep=0pt, minimum size=4pt] (v5) at (1,1) {};
\draw  (v1) edge (v2);
\draw  (v1) edge (v3);
\draw  (v1) edge (v4);
\draw  (v5) edge (v6);
\draw  (v5) edge (v2);
\draw  (v5) edge (v3);
\draw  (v7) edge (v8);
\draw  (v7) edge (v6);
\draw  (v7) edge (v2);
\draw  (v8) edge (v4);
\draw  (v8) edge (v3);
\draw  (v6) edge (v4);
\node[anchor=south,yshift=3pt,inner sep=0pt] at (v1) {1};
\node[anchor=south east,xshift=-2pt,yshift=2pt,inner sep=0pt] at (v8) {2};
\node[anchor=east,xshift=-3pt,inner sep=0pt] at (v6) {3};
\node[anchor=north east,xshift=-2pt,yshift=-2pt,inner sep=0pt] at (v2) {4};
\node[anchor=north,yshift=-3pt,inner sep=0pt] at (v3) {5};
\node[anchor=north west,xshift=2pt,yshift=-2pt,inner sep=0pt] at (v4) {6};
\node[anchor=west,xshift=3pt,inner sep=0pt] at (v7) {7};
\node[anchor=south west,xshift=2pt,yshift=2pt,inner sep=0pt] at (v5) {8};
        \end{tikzpicture}
\caption{A drawing of the Andr\'asfai graph $\Gamma_3$}
        \label{fig:andrasfai graph1}
    \end{subfigure}
\hfill
    \begin{subfigure}[b]{0.45\textwidth}
        \centering
        \begin{tikzpicture}[scale=0.85]
        \node [circle, fill=black, inner sep=0pt, minimum size=4pt] (v1) at (-0.5,1.6) {};
\node [circle, fill=black, inner sep=0pt, minimum size=4pt] (v8) at (-2,1) {};
\node [circle, fill=black, inner sep=0pt, minimum size=4pt] (v6) at (-2.6,-0.5) {};
\node [circle, fill=black, inner sep=0pt, minimum size=4pt] (v2) at (-2,-2) {};
\node [circle, fill=black, inner sep=0pt, minimum size=4pt] (v3) at (-0.5,-2.6) {};
\node [circle, fill=black, inner sep=0pt, minimum size=4pt] (v4) at (1,-2) {};
\node [circle, fill=black, inner sep=0pt, minimum size=4pt] (v7) at (1.6,-0.5) {};
\node [circle, fill=black, inner sep=0pt, minimum size=4pt] (v5) at (1,1) {};

\draw  (v1) edge (v8);
\draw  (v8) edge (v6);
\draw  (v6) edge (v2);
\draw  (v2) edge (v3);
\draw  (v3) edge (v4);
\draw  (v4) edge (v7);
\draw  (v7) edge (v5);
\draw  (v5) edge (v1);
\draw  (v1) edge (v3);
\draw  (v6) edge (v7);
\draw  (v5) edge (v2);
\draw  (v8) edge (v4);
\node[anchor=south,yshift=3pt,inner sep=0pt] at (v1) {1};
\node[anchor=south east,xshift=-2pt,yshift=2pt,inner sep=0pt] at (v8) {2};
\node[anchor=east,xshift=-3pt,inner sep=0pt] at (v6) {3};
\node[anchor=north east,xshift=-2pt,yshift=-2pt,inner sep=0pt] at (v2) {4};
\node[anchor=north,yshift=-3pt,inner sep=0pt] at (v3) {5};
\node[anchor=north west,xshift=2pt,yshift=-2pt,inner sep=0pt] at (v4) {6};
\node[anchor=west,xshift=3pt,inner sep=0pt] at (v7) {7};
\node[anchor=south west,xshift=2pt,yshift=2pt,inner sep=0pt] at (v5) {8};
        \end{tikzpicture}
        \caption{An alternative drawing of $\Gamma_3$}
        \label{fig:andrasfai graph2}
    \end{subfigure}
    \hfill 
    
    \caption{Two drawings of the Andr\'asfai graph $\Gamma_3$}
    \label{fig:andrasfai-drawings}
\end{figure}

\begin{figure}
    \centering
    \begin{subfigure}[b]{0.45\textwidth}
        \centering
        \begin{tikzpicture}[scale=0.5]
        \node [circle, fill=red, inner sep=0pt, minimum size=4pt] (v1) at (-0.5,1.5) {};
\node [circle, fill=blue, inner sep=0pt, minimum size=4pt] (v5) at (-2,1) {};
\node [circle, fill=blue, inner sep=0pt, minimum size=4pt] (v7) at (-2.5,-0.5) {};
\node [circle, fill=green, inner sep=0pt, minimum size=4pt] (v2) at (-2,-2) {};
\node [circle, fill=green, inner sep=0pt, minimum size=4pt] (v3) at (-0.5,-2.5) {};
\node [circle, fill=green, inner sep=0pt, minimum size=4pt] (v4) at (1,-2) {};
\node [circle, fill=red, inner sep=0pt, minimum size=4pt] (v6) at (1.5,-0.5) {};
\node [circle, fill=red, inner sep=0pt, minimum size=4pt] (v8) at (1,1) {};
\draw  (v1) edge (v2);
\draw  (v1) edge (v3);
\draw  (v1) edge (v4);
\draw  (v5) edge (v6);
\draw  (v5) edge (v4);
\draw  (v5) edge (v3);
\draw  (v7) edge (v8);
\draw  (v7) edge (v6);
\draw  (v7) edge (v4);
\draw  (v2) edge (v8);
\draw  (v2) edge (v6);
\draw  (v3) edge (v8);

\node [circle, fill=red, inner sep=0pt, minimum size=4pt] (v16) at (-0.5,3.5) {};
\node [circle, fill=blue, inner sep=0pt, minimum size=4pt] (v15) at (-4,1) {};
\node [circle, fill=green, inner sep=0pt, minimum size=4pt] (v9) at (3,1) {};
\node [circle, fill=blue, inner sep=0pt, minimum size=4pt] (v11) at (3,-2) {};
\node [circle, fill=green, inner sep=0pt, minimum size=4pt] (v14) at (-4,-2) {};
\node [circle, fill=red, inner sep=0pt, minimum size=4pt] (v13) at (-0.5,-4.5) {};
\node [circle, fill=black, inner sep=0pt, minimum size=4pt] (v10) at (4.5,3) {};
\node [circle, fill=black, inner sep=0pt, minimum size=4pt] (v12) at (4.5,-4) {};

\draw  (v10) edge (v11);
\draw  (v9) edge (v12);
\draw  (v11) edge (v13);
\draw  (v13) edge (v14);
\draw  (v14) edge (v15);
\draw  (v15) edge (v16);
\draw  (v16) edge (v9);
\draw  (v9) edge (v11);
\draw  (v10) edge (v12);
\draw  (v16) edge (v10);
\draw  (v12) edge (v13);
\draw  plot[smooth, tension=.7] coordinates {(v10) (-1,4) (-5,1.5) (v14)};

\draw  plot[smooth, tension=.7] coordinates {(v12) (-1,-5) (-5,-2.5) (v15)};

\node at (4.5,3.5) {$x$};
\node at (4.5,-4.5) {$y$};
\node[red] at (-0.5,-4) {$u$};
\node [blue] at (2.5,-2) {$c$};
\node[green] at (2.5,1) {$v$};
\node[red] at (-0.5,3) {$a$};
\node [blue]at (-3.5,1) {$w$};
\node [green]at (-3.5,-2) {$b$};
\node [blue]at (-2.1,1.5) {$3i-2$};
\node [red]at (-0.5,2) {$0$};
\node [red]at (2,-0.9) {$i-1$};
\node [green]at (1.5,-2) {$i$};
\node [green]at (-2.1,-2.4) {$2i-1$};
\node [blue]at (-3,-0.5) {$2i$};
\draw [dashed,blue] plot[smooth, tension=.7] coordinates {(-2.5,1) (-3,0.5) (-3,0)};
\draw [dashed,red] plot[smooth, tension=.7] coordinates {(0,2) (1,1.5) (1.5,1) (2,0)};
\draw [dashed,green] plot[smooth, tension=.7] coordinates {(1,-2.5) (-0.5,-3) (-2,-2.5)};
        \end{tikzpicture}
        \caption{The Vega graph $\Upsilon_{i}^{00}$}
        \label{fig:vega graph}
    \end{subfigure}
\hfill 
    \begin{subfigure}[b]{0.45\textwidth}
        \centering
        \begin{tikzpicture}[scale=0.7]
        \node [circle, fill=black, inner sep=0pt, minimum size=4pt] (v1) at (-0.5,2.5) {};
\node [circle, fill=black, inner sep=0pt, minimum size=4pt] (v2) at (-1.5,1) {};
\node [circle, fill=black, inner sep=0pt, minimum size=4pt] (v4) at (0.5,1) {};
\node [circle, fill=black, inner sep=0pt, minimum size=4pt] (v3) at (-0.5,-0.5) {};
\node [circle, fill=black, inner sep=0pt, minimum size=4pt] (v7) at (-0.5,-2) {};
\node [circle, fill=black, inner sep=0pt, minimum size=4pt] (v9) at (-2,-1) {};
\node [circle, fill=black, inner sep=0pt, minimum size=4pt] (v6) at (1,-1) {};
\node [circle, fill=black, inner sep=0pt, minimum size=4pt] (v5) at (2,0.5) {};
\node [circle, fill=black, inner sep=0pt, minimum size=4pt] (v11) at (-3,0.5) {};

\node [circle, fill=black, inner sep=0pt, minimum size=4pt] (v10) at (-2,-2.5) {};
\node [circle, fill=black, inner sep=0pt, minimum size=4pt] (v8) at (1,-2.5) {};
\draw  (v1) edge (v2);
\draw  (v2) edge (v3);
\draw  (v3) edge (v4);
\draw  (v4) edge (v1);
\draw  (v4) edge (v5);
\draw  (v5) edge (v6);
\draw  (v6) edge (v3);
\draw  (v3) edge (v7);
\draw  (v7) edge (v8);
\draw  (v8) edge (v6);
\draw  (v3) edge (v9);
\draw  (v9) edge (v10);
\draw  (v9) edge (v11);
\draw  (v11) edge (v2);
\draw  (v10) edge (v7);
\draw  plot[smooth, tension=.7] coordinates {(v1) (-3,1.5) (-4,-0.5) (v10)};
\draw  plot[smooth, tension=.7] coordinates {(v1) (2,1.5) (3,-0.5) (v8)};

\draw  plot[smooth, tension=.7] coordinates {(v11) (-2,3) (1,3) (v5)};

\draw  plot[smooth, tension=.7] coordinates {(v11) (-3.5,-1.5) (-2,-3.5) (v8)};

\draw  plot[smooth, tension=.7] coordinates {(v5) (2.5,-1.5) (1,-3.5) (v10)};
        \end{tikzpicture}
        \caption{The Gr\"otzsch graph $\Upsilon:=\Upsilon_{2}^{11}$}
        \label{fig:grotzsch graph}
    \end{subfigure}
    \caption{Two Vega graphs}
    \label{fig:vega-family}
\end{figure}
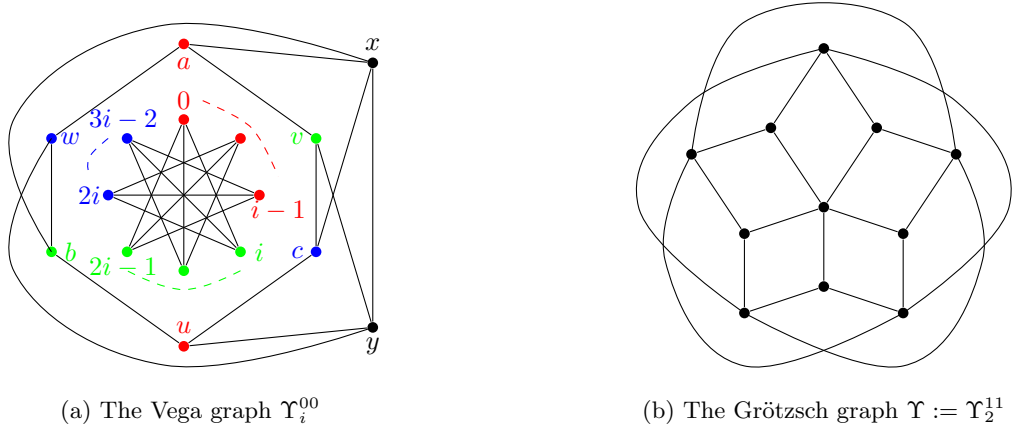
Many other related results are known; see \cite{2013-allen, 2011-lyle, 2006-luczak, 2010-luczak, 2020-o, 2002-thomassen}.
After resolving the case above $n/3$, Brandt and Thomass\'e \cite{2010-brandt} posed the following problems concerning minimum degree at most $\frac{1}{3}n$:
\begin{problem}[\cite{2010-brandt}]\label{problem1}
Determine the minimum $k$ such that every $n$-vertex triangle-free graph with minimum degree $\frac{n}{3}$ is $k$-colorable.
\end{problem}

\begin{problem}[\cite{2010-brandt}]\label{problem2}
For every $t\geq5$, determine the asymptotic behavior of the extremal deficiency
\[
s_t(n):=\frac{n}{3}-\max\{\delta(F): |V(F)|=n,\ F\text{ is triangle-free},\ \chi(F)\geq t\}.
\]
\end{problem}
\stepcounter{thm}
\L uczak and Thomass\'e \cite{2010-luczak} obtained the bound $1665$ for Problem~\ref{problem1} via VC-dimension. They also went below the $\frac{n}{3}$ threshold, proving that a $K_3$-free graph with minimum degree at least $\frac{1}{3}n-a$ has chromatic number at most $\max\{10^5,12(a+1)\}$. However, this bound goes to infinity as $a$ goes to infinity.
Recently, \L uczak, Polcyn, and Reiher \cite{2024-luczak} strengthened the theorem of Brandt and Thomass\'e. They introduced the $\mathcal{D}_{i}$ property: a graph has property $\mathcal{D}_{i}$ if, for each $k\in[i]$ and every $3k$-element multiset $H$ of $V(G)$, there exists $y\in V(G)$ such that $|N(y)\cap H|\geq k+1$, with multiplicity counted. Clearly, if $G$ satisfies $\delta(G)>\frac{n}{3}$, then $G$ has property $\mathcal{D}_{i}$ for every $i$. Their applicable $\mathcal D_3$ and $\mathcal D_4$ alternatives imply that a maximal triangle-free graph is a blow-up of an Andr\'asfai graph or a Vega graph, as stated precisely below.
\begin{theorem}[\cite{2024-luczak}]\label{thm-D3 and D4 implies blow-up of Vega graph}
    The following statements are true for a maximal triangle-free graph $G$:
    \begin{enumerate}[(1)]
        \item If $G$ does not contain $\Upsilon$ and satisfies $\mathcal{D}_{3}$, then $G$ is a blow-up of an Andr\'{a}sfai graph $\Gamma_{i}$ for some $i\geq1$.
        \item If $G$ contains $\Upsilon$ and satisfies $\mathcal{D}_{4}$, then $G$ is a blow-up of a Vega graph.
    \end{enumerate}
\end{theorem}
Combining this theorem with further analysis, we prove that every triangle-free graph with minimum degree at least $\frac{n}{3}$ is $4$-colorable, thereby solving Problem~\ref{problem1} with $k=4$. Furthermore, using a result of Brandt \cite{2001-brandt} (see also \cite{2011-lyle}), we prove the following theorem for $K_r$-free graphs with $r\geq3$.
\begin{theorem}\label{thm-main thm}
    Every $n$-vertex $K_r$-free graph with minimum degree at least $\frac{2r-5}{2r-3}n$ is $(r+1)$-colorable.
\end{theorem}

Toward Problem~\ref{problem2}, we establish a close relation between subgraphs of Kneser graphs and the chromatic number of triangle-free graphs with large minimum degree; see Theorems~\ref{thm-lower bound on chi of kneser subgraph imply k3 free graph} and~\ref{thm-sufficiency}.
Before introducing Kneser graphs and Hajnal's construction, we fix some notation. All graphs are finite and simple, and denote $[m]=\{1,\dots,m\}$ for a positive integer $m$. For a graph $F$, a vertex $v$, and a set $X\subseteq V(F)$, let $N_F(v)$ be the open neighborhood of $v$, let $N_F(X)=\bigcap_{x\in X}N_F(x)$, denote $N_X(v)=N_F(v)\cap X$ and $d_X(v)=|N_X(v)|$, and let $F[X]$ be the subgraph induced by $X$; we omit the subscript $F$ when the ambient graph is clear. A graph homomorphism $\varphi:F\to F'$ is a vertex map that sends every edge of $F$ to an edge of $F'$. A blow-up of a graph is obtained by replacing each vertex with an independent set and each edge with the complete bipartite graph between the corresponding sets.

We write $k\ll n$ when $n$ is sufficiently large relative to $k$. Given functions $f(n)$ and $g(n)$, we write $f(n)=o(g(n))$ if $f(n)/g(n)\to 0$ as $n\to\infty$, and $f(n)=\Theta(g(n))$ if there exist constants $C_1,C_2>0$ such that $C_1|g(n)|\leq |f(n)|\leq C_2|g(n)|$ for all sufficiently large $n$. Unless stated otherwise, every deficiency function $f$ used below is positive and integer-valued. We call an independent set of order $\ell$ an $\ell$-independent set, and a subgraph of order $m$ an $m$-vertex subgraph. Parameters used as cardinalities are understood to be integral. When a displayed fractional quantity is used as a cardinality, a floor or ceiling compatible with the surrounding inequality is understood; since $f\geq1$, the resulting $O(1)$ discrepancies are absorbed by the displayed constant multiples of $f$ or by $O(f)$. Similarly, $KG(a-O(f),O(f))$ denotes a Kneser graph $KG(p,q)$ with integer parameters $p=a-O(f)$ and $q=O(f)$.

We now introduce Kneser graphs and Hajnal's construction. The Kneser graph $KG(n,k)$ has vertex set $\binom{[2n+k]}{n}$, the collection of all $n$-subsets of $[2n+k]$. Two vertices are adjacent if and only if the corresponding subsets are disjoint. Thus our $KG(n,k)$ is the graph usually denoted by $KG(2n+k,n)$ in some literature. It admits a $(k+2)$-coloring, and Lov\'asz \cite{1978-lovasz} proved that this coloring is optimal. Hajnal's construction yields triangle-free graphs with minimum degree at least $(\frac{1}{3}-\varepsilon)$ times their order and arbitrarily large chromatic number, where $\varepsilon>0$ may be chosen arbitrarily small.
Given integers $k\ll n\ll \ell$ with $(2n+k)\mid \ell$, the Hajnal graph $H(k,n,\ell)$ has three vertex classes. The first induces $KG(n,k)$, the second is a $2\ell$-independent set $I_1$, and the third is an $\ell$-independent set $I_2$. Every vertex of $I_1$ is adjacent to every vertex of $I_2$. The set $I_1$ is partitioned into $2n+k$ equal parts $S_1,\dots,S_{2n+k}$. For each vertex $x\in KG(n,k)$, the vertex $x$ is adjacent to every vertex of $S_i$ if and only if $i\in x$. The Hajnal graph $H(k,n,\ell)$ has $3\ell+\binom{2n+k}{n}$ vertices, minimum degree at least $\frac{2n}{2n+k}\ell$, and chromatic number at least $k+2$.
Hajnal's construction shows that, for each fixed $t\in\mathbb{N}$, the function $s_t(n)$ in Problem~\ref{problem2} satisfies $s_t(n)=o(n)$.
Toward Problem~\ref{problem2}, we prove the Kneser-graph connection stated in Theorems~\ref{thm-lower bound on chi of kneser subgraph imply k3 free graph} and~\ref{thm-sufficiency}.
Applying Theorem \ref{thm-lower bound on chi of kneser subgraph imply k3 free graph} with a detailed setting of the parameters $\ell,n,k$ in Hajnal's construction, we obtain the following.
\begin{theorem}\label{thm-slightly improve n/logn}
Let $w(N)$ be a function that tends to infinity arbitrarily slowly as $N\to\infty$. There exists a triangle-free graph $G_N$ with minimum degree at least $\frac{1}{3}|V(G_N)|-w(N)\frac{N}{\log N}$ and chromatic number at least $w(N)$.
\end{theorem}

Combining Theorem~\ref{thm-sufficiency} below, a basic proposition on Kneser graphs, and a result relating order, odd girth, and chromatic number \cite{1984-szemeredi}, we obtain the following constant chromatic bound below the threshold $\frac{1}{3}n$.
\begin{theorem}\label{thm-break the border}
For every $0<\delta<1$ and $\e>0$, and for all sufficiently large $n$, every $n$-vertex triangle-free graph with minimum degree at least $\frac{1}{3}n-n^{1-\delta}$ has chromatic number at most
\[
10^{391}+1+\left\lceil\frac{(1+\e)(1-\delta)}{\delta}\right\rceil.
\]
\end{theorem}
With $s_t(n)$ interpreted as the extremal deficiency defined in Problem~\ref{problem2}, this theorem implies that, for every $t\geq10^{391}+3$ and $\varepsilon>0$,
\[
s_t(n)\geq n^{\,1-\frac{1+\varepsilon}{t-10^{391}-1+\varepsilon}}
\]
for all sufficiently large $n$.
Indeed, set $\delta_0=(1+\varepsilon)/(t-10^{391}-1+\varepsilon)$ and apply Theorem~\ref{thm-break the border} with any $0<\eta<\varepsilon$ in place of $\varepsilon$. Then
\[
\left\lceil\frac{(1+\eta)(1-\delta_0)}{\delta_0}\right\rceil\leq t-10^{391}-2,
\]
so every graph satisfying the corresponding minimum-degree condition is $(t-1)$-colorable.
Analyzing the proof of Theorem \ref{thm-sufficiency} more carefully, we are able to partially characterize graphs with minimum degree at least $\frac{1}{3}n-f(n)$ and large chromatic number (if such graphs exist). It is interesting that any such graph must resemble Hajnal's construction.
\begin{theorem}\label{thm-24000}
Let $f$ be a positive integer-valued function satisfying $f(n)=o(n)$, and assume that $n$ is sufficiently large. If $G$ is an $n$-vertex $K_3$-free graph with minimum degree at least $\frac{1}{3}n-f(n)$ and chromatic number at least $10^{391}$, then $G$ admits a partition $V(G)=I_1\cup I_2\cup A$ such that $I_1,I_2$ are independent, $|I_1|\geq \frac{2}{3}n-9f(n)$, $|I_2|\geq\frac{n}{3}-24f(n)$, and $G[A]$ admits a homomorphism to the Kneser graph $KG(\frac{1}{3}n-44f(n),97f(n))$.
\end{theorem}
Note that this theorem holds for every positive integer-valued function $f(n)=o(n)$ and all sufficiently large $n$; the exact form of $f(n)$ is not important. A similar result has also been obtained by Kim, Liu, Shangguan, Wang, Wu, and Xue \cite{2025-liu}. Our approaches are different.
To extend these results to $K_r$-free graphs, we prove the following theorem connecting maximal $K_r$-free graphs and $K_{r-1}$-free graphs.
\begin{theorem}\label{thm-decomposition of Kr-free graph with large minimum degree}
Fix an integer $r\geq4$, and let $f$ be a positive integer-valued function satisfying $f(n)=o(n)$. For all sufficiently large $n$, every $n$-vertex maximal $K_r$-free graph $G$ with $\delta(G)\geq \frac{2r-5}{2r-3}n-f(n)$ satisfies one of the following. All constants implicit in $O(f(n))$ may depend on $r$ only.
\begin{enumerate}
\item $G$ is the join of a nonempty independent set and a $K_{r-1}$-free graph;
\item $V(G)$ admits a partition $V(G)=I_0\cup I_1\cup \dots\cup I_{r-2}\cup A$ such that $\bigl||I_0|-\frac{n}{2r-3}\bigr|=O(f(n))$; for each $i\in [r-2]$, $\bigl||I_i|-\frac{2}{2r-3}n\bigr|=O(f(n))$; $|A|=O(f(n))$, and $G[A]$ admits a homomorphism to a Kneser graph $KG(\frac{n}{2r-3}-O(f(n)),O(f(n)))$.
\end{enumerate} 
\end{theorem}
Applying this theorem, we can easily extend Theorems \ref{thm-break the border} and \ref{thm-24000} to $K_r$-free graphs (see Theorems \ref{thm-Kr-free graph with large minimum degree has bounded chi} and \ref{thm-Kr-free case-extremal graph} below).

The remainder of the paper is organized as follows.
In Section~\ref{section-minimum degree n/3 and 4 colors}, we prove Theorem~\ref{thm-main thm}, thereby answering Problem~\ref{problem1}.
In Section~\ref{section-connection between kneser graph}, we establish the connection between $f$-vertex subgraphs of the Kneser graph $KG(n,f)$ and triangle-free graphs with large minimum degree stated in Theorems~\ref{thm-lower bound on chi of kneser subgraph imply k3 free graph} and~\ref{thm-sufficiency}. The same argument also yields Theorem~\ref{thm-24000}.
In Section~\ref{section-new upper and lower bound}, we prove Theorems~\ref{thm-slightly improve n/logn} and~\ref{thm-break the border}.
In Section~\ref{section-connect Kr-free and Kr-1-free graphs}, we prove Theorem~\ref{thm-decomposition of Kr-free graph with large minimum degree}, which relates maximal $K_r$-free graphs of large minimum degree to $K_{r-1}$-free graphs. We then derive the $K_r$-free extensions in Theorems~\ref{thm-Kr-free graph with large minimum degree has bounded chi} and~\ref{thm-Kr-free case-extremal graph}.
The final section contains remarks and open problems.

\section{Proof of Theorem~\ref{thm-main thm}}\label{section-minimum degree n/3 and 4 colors}
We first treat the triangle-free case; the $K_r$-free case will then follow from a result of Brandt \cite{2001-brandt}.
\subsection{The triangle-free case}
Let $G$ be an $n$-vertex triangle-free graph with minimum degree at least $\frac{n}{3}$. By adding edges if necessary, we may assume that $G$ is maximal triangle-free; a coloring of the resulting supergraph restricts to $G$. If $G$ satisfies $\D_4$, then Theorem~\ref{thm-D3 and D4 implies blow-up of Vega graph} implies that $G$ is either a blow-up of an Andr\'{a}sfai graph $\Gamma_i$ for some $i\geq1$ or a blow-up of a Vega graph; in either case, $G$ is $4$-colorable. Thus we may assume that $G$ violates $\D_4$. Hence there are an integer $k\in\{1,2,3,4\}$ and an indexed $3k$-element multiset $\mathcal H=(x_i)_{i\in[3k]}$ of vertices of $G$ such that every vertex of $G$ is adjacent to at most $k$ occurrences in $\mathcal H$.
Combining this with $\delta(G)\geq\frac{n}{3}$ and $d(x_1)+\dots+d(x_{3k})\geq kn$, we obtain
\begin{equation}\label{condition-exact degree of d(xi)}
    d(x_1)=\dots=d(x_{3k})=\frac{n}{3}
\end{equation}
and
\begin{equation}\label{condition-every vertex has fixed number neighbor in H}
    \text{every vertex of $G$ is adjacent to exactly $k$ occurrences in $\mathcal H$}.
\end{equation}
Define the \emph{occurrence graph} $H$ on the index set $[3k]$ by joining distinct indices $i$ and $j$ precisely when $x_ix_j\in E(G)$. Thus repeated occurrences of the same vertex of $G$ appear in $H$ as pairwise nonadjacent twins. Equation~\eqref{condition-every vertex has fixed number neighbor in H} implies that $H$ is a $k$-regular triangle-free graph on $3k$ indexed vertices. For $v\in V(G)$, let
\[
P_{\mathcal H}(v):=\{i\in[3k]:vx_i\in E(G)\}.
\]
Then $P_{\mathcal H}(v)$ is an independent $k$-set of $H$. We call $P_{\mathcal H}(v)$ the \emph{indexed occurrence-neighborhood} of $v$ and reserve $N_H(i)$ for the ordinary neighborhood of a vertex $i$ in $H$. If two vertices of $G$ are assigned the same color because their indexed occurrence-neighborhoods intersect, the common index corresponds to one original vertex $x_i$ adjacent to both, so triangle-freeness makes that color class independent. By the handshaking lemma, $k\notin\{1,3\}$.

If $k=2$, then $H$ is a $6$-vertex $2$-regular triangle-free graph and hence a copy of $C_6$. Relabel the occurrences so that $\{1,2\}\in E(H)$; then $N(x_1)$ and $N(x_2)$ are disjoint. Since every vertex is adjacent to at most one of $x_1,x_2$ and to at least one of $x_3,x_4,x_5,x_6$, we have $N(x_3)\cup N(x_4)\cup N(x_5)\cup N(x_6)\supseteq V(G)$. Hence $G$ is $4$-colorable.

If $k=4$, then $H$ is a $12$-vertex $4$-regular triangle-free graph. This graph is connected, since every component of a $4$-regular triangle-free graph has at least eight vertices. Meringer's GENREG catalogue \cite{1999-meringer} lists exactly twelve connected $12$-vertex $4$-regular graphs of girth at least $4$.
\footnote{The catalogue is available at \url{https://www.mathe2.uni-bayreuth.de/markus/REGGRAPHS/12_4_4.html}.}
Denote these graphs by $H_1,H_2,\dots,H_{12}$; see Figure~\ref{fig:twelve 12-vtx 4-reg K3-free graph}.
\begin{figure}[!htbp]
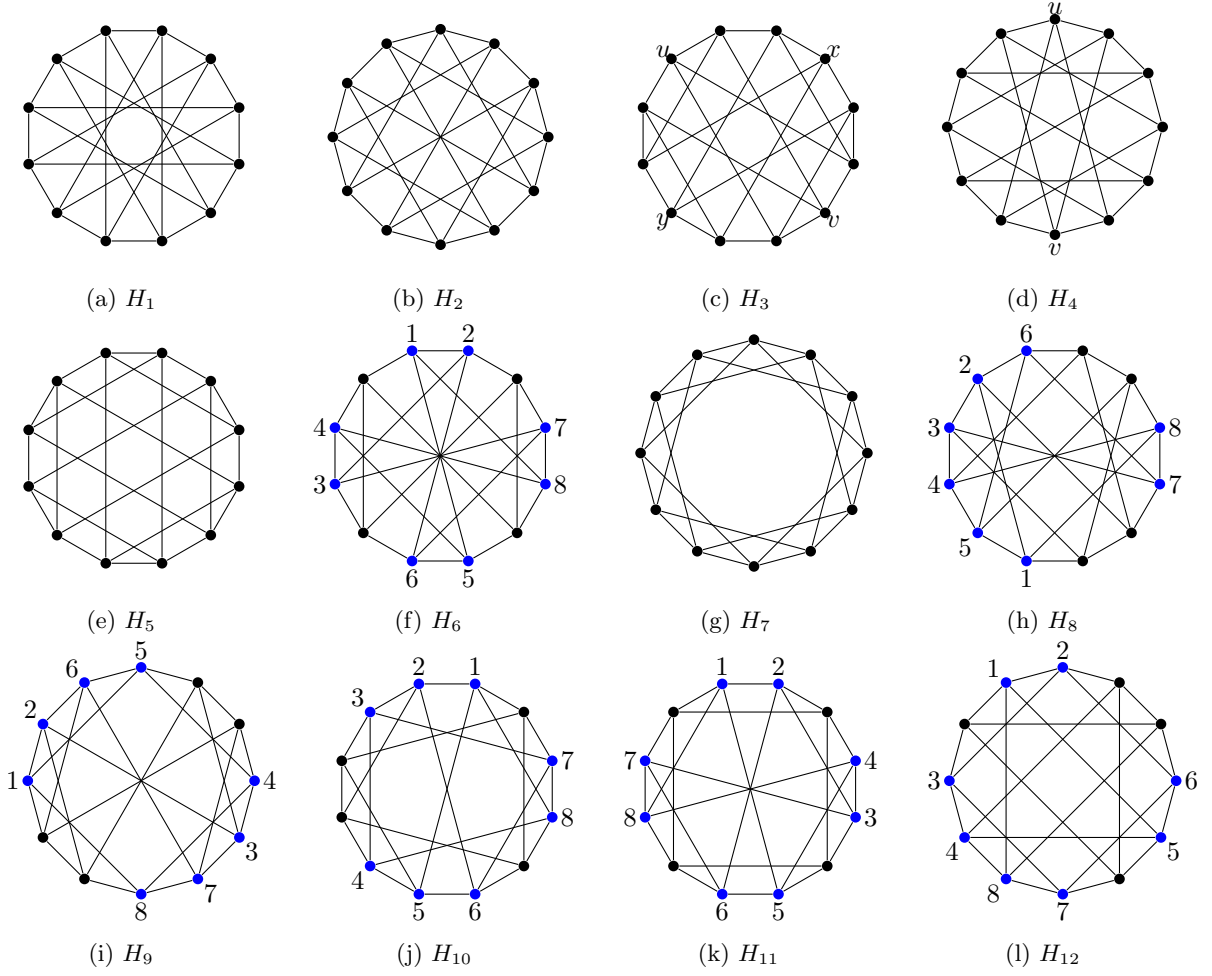

    \centering
    \begin{subfigure}[b]{0.2\textwidth}
        \centering

        \caption{$H_1$}
        \label{fig:H1}
    \end{subfigure}
    \hfill 
\centering
    \begin{subfigure}[b]{0.2\textwidth}
        \centering
        %
        \caption{$H_2$}
        \label{fig:H2}
    \end{subfigure}
    \hfill 
\centering
    \begin{subfigure}[b]{0.2\textwidth}
        \centering
        %
        \caption{$H_3$}
        \label{fig:H3}
    \end{subfigure}
    \hfill 
\centering
    \begin{subfigure}[b]{0.2\textwidth}
        \centering
        %
        \caption{$H_7$}
        \label{fig:H7}
    \end{subfigure}
    \hfill 
    \begin{subfigure}[b]{0.2\textwidth}
        \centering
        %
        \caption{$H_8$}
        \label{fig:H8}
    \end{subfigure}
    \hfill 
    \begin{subfigure}[b]{0.2\textwidth}
        \centering
        %
        \caption{$H_9$}
        \label{fig:H9}
    \end{subfigure}
    \hfill 
    \vspace*{1cm}
    \begin{subfigure}[b]{0.2\textwidth}
        \centering
        %
        \caption{$H_{10}$}
        \label{fig:H10}
    \end{subfigure}
    \hfill 
    \begin{subfigure}[b]{0.2\textwidth}
        \centering
        %
	        \caption{$H_{11}$}
        \label{fig:H11}
    \end{subfigure}
    \hfill 
    \begin{subfigure}[b]{0.2\textwidth}
        \centering
        %
        \caption{$H_{12}$}
        \label{fig:H12}
    \end{subfigure}
    \caption{Twelve $12$-vertex $4$-regular triangle-free graphs}
    \label{fig:twelve 12-vtx 4-reg K3-free graph}
\end{figure}
We divide the proof according to the isomorphism type of $H$.

Assume that $H$ is one of $H_6,H_8,H_9,H_{10},H_{11},H_{12}$. Each of these graphs contains $\Gamma_3$ as an induced subgraph: in Figure~\ref{fig:twelve 12-vtx 4-reg K3-free graph}, the blue vertices induce $\Gamma_3$ and are labeled as in Figure~\ref{fig:andrasfai graph2}.
Since $\Gamma_3$ has $8$ vertices and $\alpha(\Gamma_3)=3$, Equation~\eqref{condition-every vertex has fixed number neighbor in H} implies that each vertex of $G$ is adjacent to at most three occurrences indexed by $V(\Gamma_3)$ and hence to at least one occurrence indexed by $V(H)\setminus V(\Gamma_3)$. As $|V(H)\setminus V(\Gamma_3)|=4$ and $\bigcup_{i\in V(H)\setminus V(\Gamma_3)}N_G(x_i)\supseteq V(G)$, the graph $G$ is $4$-colorable.

Now assume that $H$ is $H_5$ or $H_7$. Both graphs are bipartite and twin-free, meaning that no two distinct vertices have the same neighborhood. In particular, the occurrence map $i\mapsto x_i$ is injective in these two cases, because repeated occurrences would be nonadjacent twins in $H$. Let $A$ and $B$ be the two parts of $H$; then $|A|=|B|=6$.
Take an arbitrary independent set $I$ in $H$ with $|I|=4$. 
If $|A\cap I|=1$ and $|B\cap I|=3$ (the case $|A\cap I|=3$ and $|B\cap I|=1$ is symmetric), then the single vertex in $A\cap I$ has $4$ neighbors in $B$, all of which lie outside $I$. This forces $|B\cap I|\le 2$, a contradiction.
If $|A\cap I|=|B\cap I|=2$, then the two vertices in $A\cap I$ have at least $5$ distinct neighbors in $B$, because their neighborhoods are different. Hence $|B\cap I|\le 1$, again a contradiction.
Thus every independent set of size $4$ in $H$ is entirely contained in $A$ or entirely contained in $B$. 
By Equation~\eqref{condition-every vertex has fixed number neighbor in H}, the indexed occurrence-neighborhood of each vertex of $G$ is an independent $4$-set of $H$. Consequently, $P_{\mathcal H}(w)$ lies entirely in $A$ or entirely in $B$ for every $w\in V(G)$.
The vertices $w$ for which $P_{\mathcal H}(w)\subseteq A$ form an independent set, because any two of their indexed occurrence-neighborhoods intersect (recall that $|A|=6$ and each has order $4$). The analogous statement holds for $B$. Hence $G$ is bipartite, contradicting its maximality because $H$ is not complete bipartite.

Assume that $H=H_3$. The graph $H_3$ is bipartite and contains twins; for example, $u$ and $v$ are twins in Figure~\ref{fig:H3}. As before, let $A$ and $B$ be the two parts of $H$, with $u,v\in A$. As in the twin-free case, every independent set $I$ of order $4$ in $H$ satisfies $|A\cap I|\in\{0,2,4\}$. If $|A\cap I|=2$, then $|B\cap I|=2$, and the vertices in $B\cap I$ have no neighbors in $A\cap I$. Since $|A\setminus I|=4$ and each vertex in $B\cap I$ has $4$ neighbors in $A$, the two vertices in $B\cap I$ must be twins; denote them by $x$ and $y$. By symmetry, the two vertices in $A\cap I$ are also twins. Moreover, $H\setminus I$ induces a $2$-regular graph, necessarily $C_8$; if it induced two disjoint copies of $C_4$, then $H$ would be a balanced $2$-blow-up of $C_6$. Thus $H_3$ contains exactly two pairs of twins, say $\{u,v\}$ and $\{x,y\}$. One can now verify that every $4$-independent set in $H$ lies entirely in $A$, lies entirely in $B$, or is $\{u,v,x,y\}$. Consequently, every vertex $w\in V(G)$ satisfies $|P_{\mathcal H}(w)\cap A|=4$, $|P_{\mathcal H}(w)\cap B|=4$, or $P_{\mathcal H}(w)=\{u,v,x,y\}$.
The vertices in each of these three classes form an independent set, because any two of their indexed occurrence-neighborhoods in $A$, in $B$, or in $\{u,v,x,y\}$, respectively, intersect. Hence $G$ is $3$-colorable.

Assume that $H=H_1$, which is the balanced $2$-blow-up of $C_6$. Figure~\ref{fig:4-independent sets in H1} lists all independent sets of order $4$ in $H$.
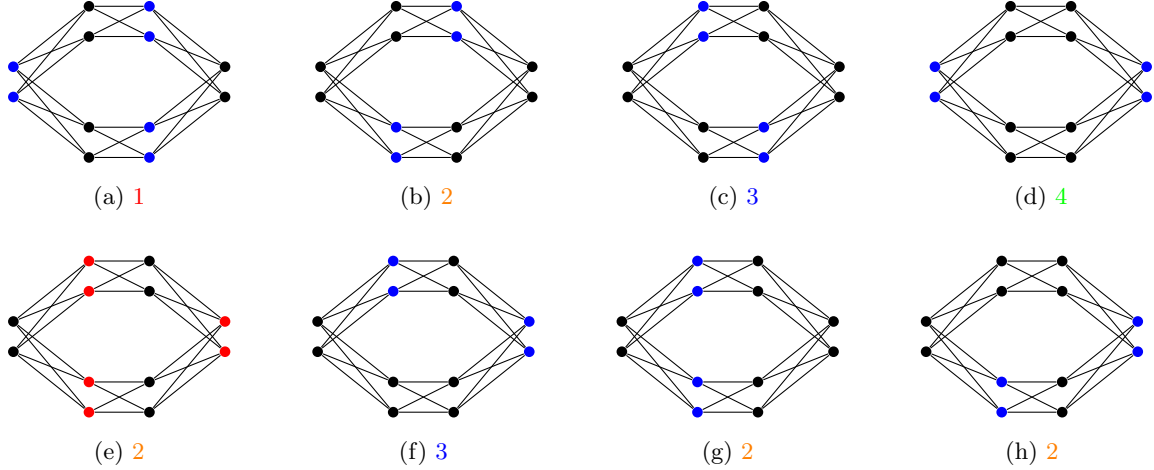
\begin{figure}[!htbp]
    \centering
    \begin{subfigure}[b]{0.2\textwidth}
        \centering
        \begin{tikzpicture}[scale=0.4]
        \node [circle, fill=black, inner sep=0pt, minimum size=4pt] (v1) at (-0.5,0.5) {};
\node [circle, fill=black, inner sep=0pt, minimum size=4pt] (v3) at (-0.5,-0.5) {};
\node [circle, fill=blue, inner sep=0pt, minimum size=4pt] (v4) at (1.5,0.5) {};
\node [circle, fill=blue, inner sep=0pt, minimum size=4pt] (v2) at (1.5,-0.5) {};
\draw  (v1) edge (v2);
\draw  (v3) edge (v4);
\draw  (v3) edge (v2);
\draw  (v1) edge (v4);
\node [circle, fill=blue, inner sep=0pt, minimum size=4pt] (v5) at (-3,-1.5) {};
\node [circle, fill=blue, inner sep=0pt, minimum size=4pt] (v6) at (-3,-2.5) {};
\node [circle, fill=black, inner sep=0pt, minimum size=4pt] (v12) at (4,-1.5) {};
\node [circle, fill=black, inner sep=0pt, minimum size=4pt] (v11) at (4,-2.5) {};
\node [circle, fill=black, inner sep=0pt, minimum size=4pt] (v8) at (-0.5,-3.5) {};
\node [circle, fill=black, inner sep=0pt, minimum size=4pt] (v7) at (-0.5,-4.5) {};
\node [circle, fill=blue, inner sep=0pt, minimum size=4pt] (v10) at (1.5,-3.5) {};
\node [circle, fill=blue, inner sep=0pt, minimum size=4pt] (v9) at (1.5,-4.5) {};
\draw  (v1) edge (v5);
\draw  (v1) edge (v6);
\draw  (v3) edge (v5);
\draw  (v3) edge (v6);
\draw  (v6) edge (v7);
\draw  (v6) edge (v8);
\draw  (v5) edge (v8);
\draw  (v5) edge (v7);
\draw  (v7) edge (v9);
\draw  (v7) edge (v10);
\draw  (v8) edge (v10);
\draw  (v8) edge (v9);
\draw  (v9) edge (v11);
\draw  (v9) edge (v12);
\draw  (v10) edge (v11);
\draw  (v10) edge (v12);
\draw  (v12) edge (v4);
\draw  (v4) edge (v11);
\draw  (v2) edge (v12);
\draw  (v2) edge (v11);
        \end{tikzpicture}
        \caption{\color{red}{1}}
        \label{fig:color of H1, in first part}
    \end{subfigure}
    \hfill 
    \begin{subfigure}[b]{0.2\textwidth}
        \centering
        \begin{tikzpicture}[scale=0.4]
        \node [circle, fill=black, inner sep=0pt, minimum size=4pt] (v1) at (-0.5,0.5) {};
\node [circle, fill=black, inner sep=0pt, minimum size=4pt] (v3) at (-0.5,-0.5) {};
\node [circle, fill=blue, inner sep=0pt, minimum size=4pt] (v4) at (1.5,0.5) {};
\node [circle, fill=blue, inner sep=0pt, minimum size=4pt] (v2) at (1.5,-0.5) {};
\draw  (v1) edge (v2);
\draw  (v3) edge (v4);
\draw  (v3) edge (v2);
\draw  (v1) edge (v4);
\node [circle, fill=black, inner sep=0pt, minimum size=4pt] (v5) at (-3,-1.5) {};
\node [circle, fill=black, inner sep=0pt, minimum size=4pt] (v6) at (-3,-2.5) {};
\node [circle, fill=black, inner sep=0pt, minimum size=4pt] (v12) at (4,-1.5) {};
\node [circle, fill=black, inner sep=0pt, minimum size=4pt] (v11) at (4,-2.5) {};
\node [circle, fill=blue, inner sep=0pt, minimum size=4pt] (v8) at (-0.5,-3.5) {};
\node [circle, fill=blue, inner sep=0pt, minimum size=4pt] (v7) at (-0.5,-4.5) {};
\node [circle, fill=black, inner sep=0pt, minimum size=4pt] (v10) at (1.5,-3.5) {};
\node [circle, fill=black, inner sep=0pt, minimum size=4pt] (v9) at (1.5,-4.5) {};
\draw  (v1) edge (v5);
\draw  (v1) edge (v6);
\draw  (v3) edge (v5);
\draw  (v3) edge (v6);
\draw  (v6) edge (v7);
\draw  (v6) edge (v8);
\draw  (v5) edge (v8);
\draw  (v5) edge (v7);
\draw  (v7) edge (v9);
\draw  (v7) edge (v10);
\draw  (v8) edge (v10);
\draw  (v8) edge (v9);
\draw  (v9) edge (v11);
\draw  (v9) edge (v12);
\draw  (v10) edge (v11);
\draw  (v10) edge (v12);
\draw  (v12) edge (v4);
\draw  (v4) edge (v11);
\draw  (v2) edge (v12);
\draw  (v2) edge (v11);
        \end{tikzpicture}
        \caption{\color{orange}{2}}
        \label{fig:color of H1, diagonal1}
    \end{subfigure}
    \hfill 
\begin{subfigure}[b]{0.2\textwidth}
        \centering
        \begin{tikzpicture}[scale=0.4]
        \node [circle, fill=blue, inner sep=0pt, minimum size=4pt] (v1) at (-0.5,0.5) {};
\node [circle, fill=blue, inner sep=0pt, minimum size=4pt] (v3) at (-0.5,-0.5) {};
\node [circle, fill=black, inner sep=0pt, minimum size=4pt] (v4) at (1.5,0.5) {};
\node [circle, fill=black, inner sep=0pt, minimum size=4pt] (v2) at (1.5,-0.5) {};
\draw  (v1) edge (v2);
\draw  (v3) edge (v4);
\draw  (v3) edge (v2);
\draw  (v1) edge (v4);
\node [circle, fill=black, inner sep=0pt, minimum size=4pt] (v5) at (-3,-1.5) {};
\node [circle, fill=black, inner sep=0pt, minimum size=4pt] (v6) at (-3,-2.5) {};
\node [circle, fill=black, inner sep=0pt, minimum size=4pt] (v12) at (4,-1.5) {};
\node [circle, fill=black, inner sep=0pt, minimum size=4pt] (v11) at (4,-2.5) {};
\node [circle, fill=black, inner sep=0pt, minimum size=4pt] (v8) at (-0.5,-3.5) {};
\node [circle, fill=black, inner sep=0pt, minimum size=4pt] (v7) at (-0.5,-4.5) {};
\node [circle, fill=blue, inner sep=0pt, minimum size=4pt] (v10) at (1.5,-3.5) {};
\node [circle, fill=blue, inner sep=0pt, minimum size=4pt] (v9) at (1.5,-4.5) {};
\draw  (v1) edge (v5);
\draw  (v1) edge (v6);
\draw  (v3) edge (v5);
\draw  (v3) edge (v6);
\draw  (v6) edge (v7);
\draw  (v6) edge (v8);
\draw  (v5) edge (v8);
\draw  (v5) edge (v7);
\draw  (v7) edge (v9);
\draw  (v7) edge (v10);
\draw  (v8) edge (v10);
\draw  (v8) edge (v9);
\draw  (v9) edge (v11);
\draw  (v9) edge (v12);
\draw  (v10) edge (v11);
\draw  (v10) edge (v12);
\draw  (v12) edge (v4);
\draw  (v4) edge (v11);
\draw  (v2) edge (v12);
\draw  (v2) edge (v11);
        \end{tikzpicture}
        \caption{\color{blue}{3}}
        \label{fig:color of H1, diagonal2}
    \end{subfigure}
    \hfill 
    \begin{subfigure}[b]{0.2\textwidth}
        \centering
        \begin{tikzpicture}[scale=0.4]
        \node [circle, fill=black, inner sep=0pt, minimum size=4pt] (v1) at (-0.5,0.5) {};
\node [circle, fill=black, inner sep=0pt, minimum size=4pt] (v3) at (-0.5,-0.5) {};
\node [circle, fill=black, inner sep=0pt, minimum size=4pt] (v4) at (1.5,0.5) {};
\node [circle, fill=black, inner sep=0pt, minimum size=4pt] (v2) at (1.5,-0.5) {};
\draw  (v1) edge (v2);
\draw  (v3) edge (v4);
\draw  (v3) edge (v2);
\draw  (v1) edge (v4);
\node [circle, fill=blue, inner sep=0pt, minimum size=4pt] (v5) at (-3,-1.5) {};
\node [circle, fill=blue, inner sep=0pt, minimum size=4pt] (v6) at (-3,-2.5) {};
\node [circle, fill=blue, inner sep=0pt, minimum size=4pt] (v12) at (4,-1.5) {};
\node [circle, fill=blue, inner sep=0pt, minimum size=4pt] (v11) at (4,-2.5) {};
\node [circle, fill=black, inner sep=0pt, minimum size=4pt] (v8) at (-0.5,-3.5) {};
\node [circle, fill=black, inner sep=0pt, minimum size=4pt] (v7) at (-0.5,-4.5) {};
\node [circle, fill=black, inner sep=0pt, minimum size=4pt] (v10) at (1.5,-3.5) {};
\node [circle, fill=black, inner sep=0pt, minimum size=4pt] (v9) at (1.5,-4.5) {};
\draw  (v1) edge (v5);
\draw  (v1) edge (v6);
\draw  (v3) edge (v5);
\draw  (v3) edge (v6);
\draw  (v6) edge (v7);
\draw  (v6) edge (v8);
\draw  (v5) edge (v8);
\draw  (v5) edge (v7);
\draw  (v7) edge (v9);
\draw  (v7) edge (v10);
\draw  (v8) edge (v10);
\draw  (v8) edge (v9);
\draw  (v9) edge (v11);
\draw  (v9) edge (v12);
\draw  (v10) edge (v11);
\draw  (v10) edge (v12);
\draw  (v12) edge (v4);
\draw  (v4) edge (v11);
\draw  (v2) edge (v12);
\draw  (v2) edge (v11);
        \end{tikzpicture}
        \caption{\color{green}{4}}
        \label{fig:color of H1, diagonal3}
    \end{subfigure}
    \hfill\\ 
    \vspace*{0.5cm}
\begin{subfigure}[b]{0.2\textwidth}
        \centering
        \begin{tikzpicture}[scale=0.4]
        \node [circle, fill=red, inner sep=0pt, minimum size=4pt] (v1) at (-0.5,0.5) {};
\node [circle, fill=red, inner sep=0pt, minimum size=4pt] (v3) at (-0.5,-0.5) {};
\node [circle, fill=black, inner sep=0pt, minimum size=4pt] (v4) at (1.5,0.5) {};
\node [circle, fill=black, inner sep=0pt, minimum size=4pt] (v2) at (1.5,-0.5) {};
\draw  (v1) edge (v2);
\draw  (v3) edge (v4);
\draw  (v3) edge (v2);
\draw  (v1) edge (v4);
\node [circle, fill=black, inner sep=0pt, minimum size=4pt] (v5) at (-3,-1.5) {};
\node [circle, fill=black, inner sep=0pt, minimum size=4pt] (v6) at (-3,-2.5) {};
\node [circle, fill=red, inner sep=0pt, minimum size=4pt] (v12) at (4,-1.5) {};
\node [circle, fill=red, inner sep=0pt, minimum size=4pt] (v11) at (4,-2.5) {};
\node [circle, fill=red, inner sep=0pt, minimum size=4pt] (v8) at (-0.5,-3.5) {};
\node [circle, fill=red, inner sep=0pt, minimum size=4pt] (v7) at (-0.5,-4.5) {};
\node [circle, fill=black, inner sep=0pt, minimum size=4pt] (v10) at (1.5,-3.5) {};
\node [circle, fill=black, inner sep=0pt, minimum size=4pt] (v9) at (1.5,-4.5) {};
\draw  (v1) edge (v5);
\draw  (v1) edge (v6);
\draw  (v3) edge (v5);
\draw  (v3) edge (v6);
\draw  (v6) edge (v7);
\draw  (v6) edge (v8);
\draw  (v5) edge (v8);
\draw  (v5) edge (v7);
\draw  (v7) edge (v9);
\draw  (v7) edge (v10);
\draw  (v8) edge (v10);
\draw  (v8) edge (v9);
\draw  (v9) edge (v11);
\draw  (v9) edge (v12);
\draw  (v10) edge (v11);
\draw  (v10) edge (v12);
\draw  (v12) edge (v4);
\draw  (v4) edge (v11);
\draw  (v2) edge (v12);
\draw  (v2) edge (v11);
        \end{tikzpicture}
        \caption{\color{orange}{2}}
        \label{fig:color of H1, in second part, crossing}
    \end{subfigure}
    \hfill 
    \begin{subfigure}[b]{0.2\textwidth}
        \centering
        \begin{tikzpicture}[scale=0.4]
        \node [circle, fill=blue, inner sep=0pt, minimum size=4pt] (v1) at (-0.5,0.5) {};
\node [circle, fill=blue, inner sep=0pt, minimum size=4pt] (v3) at (-0.5,-0.5) {};
\node [circle, fill=black, inner sep=0pt, minimum size=4pt] (v4) at (1.5,0.5) {};
\node [circle, fill=black, inner sep=0pt, minimum size=4pt] (v2) at (1.5,-0.5) {};
\draw  (v1) edge (v2);
\draw  (v3) edge (v4);
\draw  (v3) edge (v2);
\draw  (v1) edge (v4);
\node [circle, fill=black, inner sep=0pt, minimum size=4pt] (v5) at (-3,-1.5) {};
\node [circle, fill=black, inner sep=0pt, minimum size=4pt] (v6) at (-3,-2.5) {};
\node [circle, fill=blue, inner sep=0pt, minimum size=4pt] (v12) at (4,-1.5) {};
\node [circle, fill=blue, inner sep=0pt, minimum size=4pt] (v11) at (4,-2.5) {};
\node [circle, fill=black, inner sep=0pt, minimum size=4pt] (v8) at (-0.5,-3.5) {};
\node [circle, fill=black, inner sep=0pt, minimum size=4pt] (v7) at (-0.5,-4.5) {};
\node [circle, fill=black, inner sep=0pt, minimum size=4pt] (v10) at (1.5,-3.5) {};
\node [circle, fill=black, inner sep=0pt, minimum size=4pt] (v9) at (1.5,-4.5) {};
\draw  (v1) edge (v5);
\draw  (v1) edge (v6);
\draw  (v3) edge (v5);
\draw  (v3) edge (v6);
\draw  (v6) edge (v7);
\draw  (v6) edge (v8);
\draw  (v5) edge (v8);
\draw  (v5) edge (v7);
\draw  (v7) edge (v9);
\draw  (v7) edge (v10);
\draw  (v8) edge (v10);
\draw  (v8) edge (v9);
\draw  (v9) edge (v11);
\draw  (v9) edge (v12);
\draw  (v10) edge (v11);
\draw  (v10) edge (v12);
\draw  (v12) edge (v4);
\draw  (v4) edge (v11);
\draw  (v2) edge (v12);
\draw  (v2) edge (v11);
        \end{tikzpicture}
        \caption{\color{blue}{3}}
        \label{fig:color of H1, in second part, two vertex1}
    \end{subfigure}
    \hfill 
    \begin{subfigure}[b]{0.2\textwidth}
        \centering
        \begin{tikzpicture}[scale=0.4]
        \node [circle, fill=blue, inner sep=0pt, minimum size=4pt] (v1) at (-0.5,0.5) {};
\node [circle, fill=blue, inner sep=0pt, minimum size=4pt] (v3) at (-0.5,-0.5) {};
\node [circle, fill=black, inner sep=0pt, minimum size=4pt] (v4) at (1.5,0.5) {};
\node [circle, fill=black, inner sep=0pt, minimum size=4pt] (v2) at (1.5,-0.5) {};
\draw  (v1) edge (v2);
\draw  (v3) edge (v4);
\draw  (v3) edge (v2);
\draw  (v1) edge (v4);
\node [circle, fill=black, inner sep=0pt, minimum size=4pt] (v5) at (-3,-1.5) {};
\node [circle, fill=black, inner sep=0pt, minimum size=4pt] (v6) at (-3,-2.5) {};
\node [circle, fill=black, inner sep=0pt, minimum size=4pt] (v12) at (4,-1.5) {};
\node [circle, fill=black, inner sep=0pt, minimum size=4pt] (v11) at (4,-2.5) {};
\node [circle, fill=blue, inner sep=0pt, minimum size=4pt] (v8) at (-0.5,-3.5) {};
\node [circle, fill=blue, inner sep=0pt, minimum size=4pt] (v7) at (-0.5,-4.5) {};
\node [circle, fill=black, inner sep=0pt, minimum size=4pt] (v10) at (1.5,-3.5) {};
\node [circle, fill=black, inner sep=0pt, minimum size=4pt] (v9) at (1.5,-4.5) {};
\draw  (v1) edge (v5);
\draw  (v1) edge (v6);
\draw  (v3) edge (v5);
\draw  (v3) edge (v6);
\draw  (v6) edge (v7);
\draw  (v6) edge (v8);
\draw  (v5) edge (v8);
\draw  (v5) edge (v7);
\draw  (v7) edge (v9);
\draw  (v7) edge (v10);
\draw  (v8) edge (v10);
\draw  (v8) edge (v9);
\draw  (v9) edge (v11);
\draw  (v9) edge (v12);
\draw  (v10) edge (v11);
\draw  (v10) edge (v12);
\draw  (v12) edge (v4);
\draw  (v4) edge (v11);
\draw  (v2) edge (v12);
\draw  (v2) edge (v11);
        \end{tikzpicture}
        \caption{\color{orange}{2}}
        \label{fig:color of H1, in second part, two vertex2}
    \end{subfigure}
    \hfill 
    \begin{subfigure}[b]{0.2\textwidth}
        \centering
        \begin{tikzpicture}[scale=0.4]
        \node [circle, fill=black, inner sep=0pt, minimum size=4pt] (v1) at (-0.5,0.5) {};
\node [circle, fill=black, inner sep=0pt, minimum size=4pt] (v3) at (-0.5,-0.5) {};
\node [circle, fill=black, inner sep=0pt, minimum size=4pt] (v4) at (1.5,0.5) {};
\node [circle, fill=black, inner sep=0pt, minimum size=4pt] (v2) at (1.5,-0.5) {};
\draw  (v1) edge (v2);
\draw  (v3) edge (v4);
\draw  (v3) edge (v2);
\draw  (v1) edge (v4);
\node [circle, fill=black, inner sep=0pt, minimum size=4pt] (v5) at (-3,-1.5) {};
\node [circle, fill=black, inner sep=0pt, minimum size=4pt] (v6) at (-3,-2.5) {};
\node [circle, fill=blue, inner sep=0pt, minimum size=4pt] (v12) at (4,-1.5) {};
\node [circle, fill=blue, inner sep=0pt, minimum size=4pt] (v11) at (4,-2.5) {};
\node [circle, fill=blue, inner sep=0pt, minimum size=4pt] (v8) at (-0.5,-3.5) {};
\node [circle, fill=blue, inner sep=0pt, minimum size=4pt] (v7) at (-0.5,-4.5) {};
\node [circle, fill=black, inner sep=0pt, minimum size=4pt] (v10) at (1.5,-3.5) {};
\node [circle, fill=black, inner sep=0pt, minimum size=4pt] (v9) at (1.5,-4.5) {};
\draw  (v1) edge (v5);
\draw  (v1) edge (v6);
\draw  (v3) edge (v5);
\draw  (v3) edge (v6);
\draw  (v6) edge (v7);
\draw  (v6) edge (v8);
\draw  (v5) edge (v8);
\draw  (v5) edge (v7);
\draw  (v7) edge (v9);
\draw  (v7) edge (v10);
\draw  (v8) edge (v10);
\draw  (v8) edge (v9);
\draw  (v9) edge (v11);
\draw  (v9) edge (v12);
\draw  (v10) edge (v11);
\draw  (v10) edge (v12);
\draw  (v12) edge (v4);
\draw  (v4) edge (v11);
\draw  (v2) edge (v12);
\draw  (v2) edge (v11);
        \end{tikzpicture}
        \caption{\color{orange}{2}}
        \label{fig:color of H1, in second part, two vertex3}
    \end{subfigure}
    \hfill 
    
    \caption{The $4$-independent sets in $H_1$; colored vertices indicate the sets}
    \label{fig:4-independent sets in H1}
\end{figure}
In Figure~\ref{fig:color of H1, in first part}, the blue vertices form one part $A$ of the bipartite graph $H_{1}$, while in Figure~\ref{fig:color of H1, in second part, crossing}, the red vertices form the other part $B$. The blue vertices in Figure~\ref{fig:color of H1, in first part} indicate all independent sets of order $4$ contained in $A$. The red vertices in Figure~\ref{fig:color of H1, in second part, crossing} indicate those contained in $B$ that contain exactly one pair of twins, while Figures~\ref{fig:color of H1, in second part, two vertex1}--\ref{fig:color of H1, in second part, two vertex3} indicate those containing exactly two pairs of twins. If an independent set $I$ of order $4$ is contained in neither $A$ nor $B$, then it is one of the sets shown in Figures~\ref{fig:color of H1, diagonal1}--\ref{fig:color of H1, diagonal3}. These eight displayed families cover all indexed occurrence-neighborhoods. Assign each vertex of $G$ to one eligible family, breaking ties by the order in the list below, and color the resulting vertex classes as follows.
\begin{enumerate}
    \item Color the vertices whose indexed occurrence-neighborhood is the $4$-independent set in Figure \ref{fig:color of H1, in first part} by {\color{red}{1}}.
    \item Color the vertices whose indexed occurrence-neighborhood is a $4$-independent set in one of Figures \ref{fig:color of H1, diagonal1}, \ref{fig:color of H1, in second part, crossing}, \ref{fig:color of H1, in second part, two vertex2}, \ref{fig:color of H1, in second part, two vertex3} by {\color{orange}{2}}.
    \item Color the vertices whose indexed occurrence-neighborhood is a $4$-independent set in one of Figures \ref{fig:color of H1, diagonal2}, \ref{fig:color of H1, in second part, two vertex1} by {\color{blue}{3}}.
    \item Color the vertices whose indexed occurrence-neighborhood is the $4$-independent set in Figure \ref{fig:color of H1, diagonal3} by {\color{green}{4}}.
\end{enumerate}
Any two vertices assigned the same color have intersecting indexed occurrence-neighborhoods and hence a common neighbor among the vertices $x_i$. Thus $G$ is $4$-colorable.

Assume that $H=H_{2}$. Figure~\ref{fig:color of H2} similarly lists all $4$-independent sets. In each subfigure, the highlighted vertices encode the corresponding family. Figures~\ref{fig:color of H2-f}--\ref{fig:color of H2-i} each highlight five vertices: Figure~\ref{fig:color of H2-f} uses four blue vertices and one red vertex, whereas Figures~\ref{fig:color of H2-g}--\ref{fig:color of H2-i} use five blue vertices. In Figure~\ref{fig:color of H2-f}, the red vertex distinguishes the $4$-independent subsets that contain it; Figure~\ref{fig:color of H2-m} represents the corresponding subsets that omit it.
Observe that the $4$-independent sets represented by Figures \ref{fig:color of H2-j}, \ref{fig:color of H2-k}, \ref{fig:color of H2-l}, and \ref{fig:color of H2-m} are respectively contained in the $5$-vertex sets represented by Figures \ref{fig:color of H2-h}, \ref{fig:color of H2-g}, \ref{fig:color of H2-i}, and \ref{fig:color of H2-f}.
We now show that Figures \ref{fig:color of H2-a}--\ref{fig:color of H2-i} together with Figure \ref{fig:color of H2-m} give exactly all $4$-independent sets of $H$. We partition the $4$-independent sets according to the size of their intersection with the four blue vertices shown in Figure \ref{fig:color of H2-a}; denote this set of four blue vertices by $I_{0}$. If a $4$-independent set intersects $I_{0}$ in at least three vertices, then it must be $I_{0}$ itself. If it intersects $I_{0}$ in exactly two vertices, then it is one of the independent sets depicted in Figures \ref{fig:color of H2-b}, \ref{fig:color of H2-c}, \ref{fig:color of H2-d}, and \ref{fig:color of H2-e}. If it intersects $I_{0}$ in exactly one vertex, then it must be among the independent sets contained in the $5$-vertex sets of Figures \ref{fig:color of H2-f}, \ref{fig:color of H2-g}, \ref{fig:color of H2-h}, and \ref{fig:color of H2-i}. If it is disjoint from $I_{0}$, then $H\setminus I_{0}$ consists of two vertex-disjoint $C_{4}$'s, which yields exactly the four possibilities shown in Figures \ref{fig:color of H2-j}, \ref{fig:color of H2-k}, \ref{fig:color of H2-l}, and \ref{fig:color of H2-m}, all of which are already contained in the aforementioned $5$-vertex sets. This completes the characterization of all $4$-independent sets in $H$.
Recall that Figure \ref{fig:color of H2-f} represents those $4$-independent sets in its $5$-vertex set that contain the red vertex, while Figure \ref{fig:color of H2-m} represents those in the same $5$-vertex set that do not contain the red vertex. Consequently, Figures \ref{fig:color of H2-a}--\ref{fig:color of H2-i} together with Figure \ref{fig:color of H2-m} indeed give all $4$-independent sets of $H$.
Thus the ten displayed families cover all indexed occurrence-neighborhoods. Assign each vertex of $G$ to one eligible family, breaking ties by the order in the coloring list below.
\begin{figure}[!htbp]
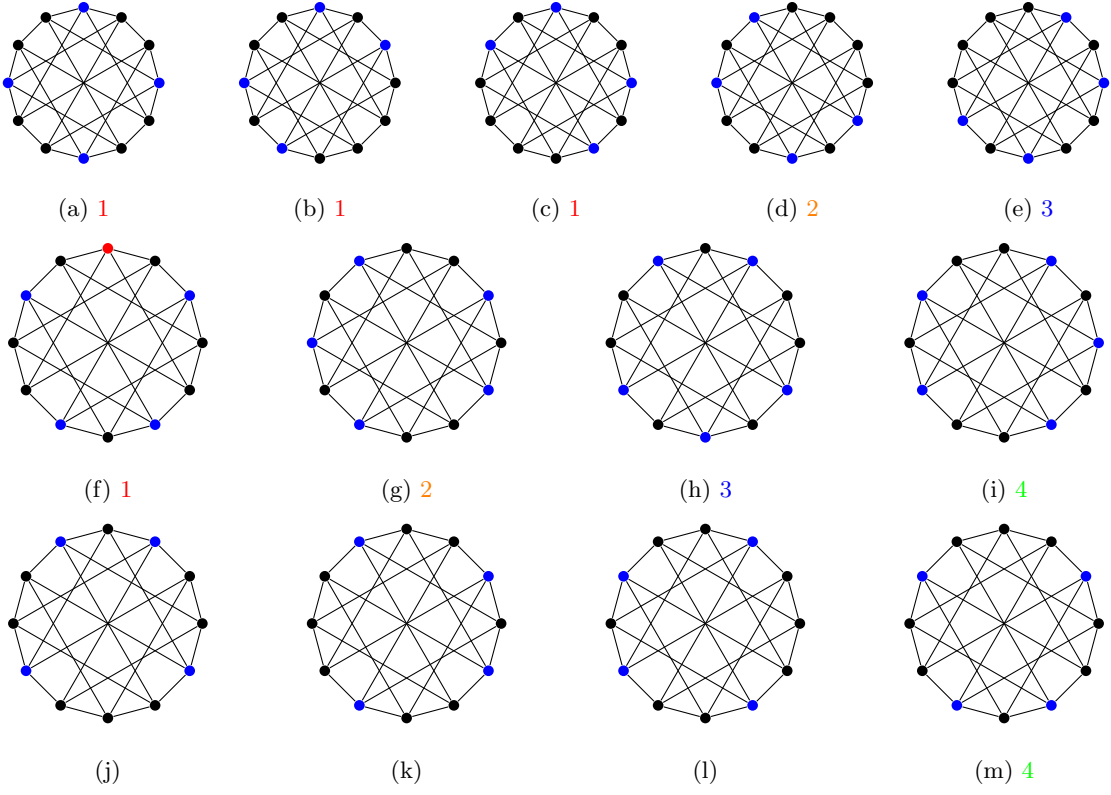

    \centering
    \begin{subfigure}[b]{0.18\textwidth}
        \centering

        \caption{\color{red}{1}}
        \label{fig:color of H2-a}
    \end{subfigure}
    \hfill 
    \begin{subfigure}[b]{0.18\textwidth}
        \centering
        %
        \caption{\color{red}{1}}
        \label{fig:color of H2-b}
    \end{subfigure}
\hfill
\begin{subfigure}[b]{0.18\textwidth}
        \centering
        %
        \caption{\color{red}{1}}
        \label{fig:color of H2-c}
    \end{subfigure}
    \hfill 
    \begin{subfigure}[b]{0.18\textwidth}
        \centering
        %
        \caption{\color{orange}{2}}
        \label{fig:color of H2-d}
    \end{subfigure}
\hfill
\begin{subfigure}[b]{0.18\textwidth}
        \centering
        %
        \caption{\color{blue}{3}}
        \label{fig:color of H2-e}
    \end{subfigure}
\hfill
    \begin{subfigure}[b]{0.222\textwidth}
        \centering
        %
        \caption{\color{red}{1}}
        \label{fig:color of H2-f}
    \end{subfigure}
\hfill
\begin{subfigure}[b]{0.222\textwidth}
        \centering
        %
        \caption{\color{orange}{2}}
        \label{fig:color of H2-g}
    \end{subfigure}
    \hfill 
    \begin{subfigure}[b]{0.222\textwidth}
        \centering
        %
        \caption{\color{blue}{3}}
        \label{fig:color of H2-h}
    \end{subfigure}
\hfill
\begin{subfigure}[b]{0.222\textwidth}
        \centering
        %
        \caption{\color{green}{4}}
        \label{fig:color of H2-i}
    \end{subfigure}
    \hfill 
    \begin{subfigure}[b]{0.222\textwidth}
        \centering
        %
        \caption{ }
        \label{fig:color of H2-j}
    \end{subfigure}
\hfill
\begin{subfigure}[b]{0.222\textwidth}
        \centering
        %
        \caption{ }
        \label{fig:color of H2-k}
    \end{subfigure}
    \hfill 
    \begin{subfigure}[b]{0.222\textwidth}
        \centering
        %
        \caption{ }
        \label{fig:color of H2-l}
    \end{subfigure}
\hfill 
    \begin{subfigure}[b]{0.222\textwidth}
        \centering
        %
        \caption{\color{green}{4}}
        \label{fig:color of H2-m}
    \end{subfigure}
    \caption{The $4$-independent sets in $H_2$}
    \label{fig:color of H2}
\end{figure}
Color the corresponding vertex classes as follows.
\begin{enumerate}
    \item Color the vertices whose indexed occurrence-neighborhood is a $4$-independent set in one of Figures \ref{fig:color of H2-a}, \ref{fig:color of H2-b}, \ref{fig:color of H2-c}, \ref{fig:color of H2-f} by {\color{red}{1}}.
    \item Color the vertices whose indexed occurrence-neighborhood is a $4$-independent set in one of Figures \ref{fig:color of H2-d}, \ref{fig:color of H2-g} by {\color{orange}{2}}.
    \item Color the vertices whose indexed occurrence-neighborhood is a $4$-independent set in one of Figures \ref{fig:color of H2-e}, \ref{fig:color of H2-h} by {\color{blue}{3}}.
    \item Color the vertices whose indexed occurrence-neighborhood is a $4$-independent set in one of Figures \ref{fig:color of H2-i}, \ref{fig:color of H2-m} by {\color{green}{4}}.
\end{enumerate}
Any two vertices assigned the same color have intersecting indexed occurrence-neighborhoods and hence a common neighbor among the vertices $x_i$. Thus $G$ is $4$-colorable.

Assume that $H=H_{4}$. We first show that every $4$-independent set of $H$ is contained in one of the sets indicated by blue vertices in Figure~\ref{fig:color of H4}. Draw $H$ as in Figure~\ref{fig:H4}. The top vertex $u$ and the bottom vertex $v$ are twins, so every maximal independent set contains either both or neither. There are only three maximal independent sets containing both $u$ and $v$, as shown in Figures~\ref{fig:color of H4-a}, \ref{fig:color of H4-b}, and~\ref{fig:color of H4-c}.
Next we consider maximal independent sets of size at least $4$ that contain neither $u$ nor $v$. Let $N_{H}(u)$ denote the set of neighbors of $u$ in $H$. Since $V(H)\setminus (\{u,v\}\cup N_{H}(u))$ contains a $C_{6}$ and its independence number is at most $3$, any maximal independent set of size at least $4$ containing none of $\{u,v\}$ must intersect $N_{H}(u)$ in at least one vertex. The unique maximal independent set that contains all of $N_{H}(u)$ is $N_{H}(u)$ itself, as shown in Figure \ref{fig:color of H4-d}.
If a maximal independent set of size at least $4$ intersects $N_{H}(u)$ in exactly three vertices, then its remaining vertex is also uniquely determined, as shown in Figures \ref{fig:color of H4-e}, \ref{fig:color of H4-f}, \ref{fig:color of H4-g}, and \ref{fig:color of H4-h}.
If a maximal independent set of size at least $4$ intersects $N_{H}(u)$ in exactly two vertices, then one can easily verify that these two vertices, together with the corresponding maximal independent sets, must be as depicted in Figures \ref{fig:color of H4-i} and \ref{fig:color of H4-j}.
Finally, one can check that any independent set of size at least $4$ intersecting $N_{H}(u)$ in exactly one vertex is contained in one of the configurations shown in Figures~\ref{fig:color of H4-i} and~\ref{fig:color of H4-j}. This completes the characterization of all $4$-independent sets of $H_{4}$. The ten displayed families therefore cover all indexed occurrence-neighborhoods. Assign each vertex of $G$ to one eligible family, breaking ties by the order in the coloring list below.
Then
\begin{enumerate}
    \item Color the vertices whose indexed occurrence-neighborhood is a $4$-independent set in one of Figures \ref{fig:color of H4-a}, \ref{fig:color of H4-c} by {\color{red}{1}}.
    \item Color the vertices whose indexed occurrence-neighborhood is the $4$-independent set in Figure \ref{fig:color of H4-b} by {\color{orange}{2}}.
    \item Color the vertices whose indexed occurrence-neighborhood is a $4$-independent set in one of Figures \ref{fig:color of H4-d}, \ref{fig:color of H4-f}, \ref{fig:color of H4-g}, \ref{fig:color of H4-i} by {\color{blue}{3}}.
    \item Color the vertices whose indexed occurrence-neighborhood is a $4$-independent set in one of Figures \ref{fig:color of H4-e}, \ref{fig:color of H4-h}, \ref{fig:color of H4-j} by {\color{green}{4}}.
\end{enumerate}
Any two vertices assigned the same color have intersecting indexed occurrence-neighborhoods and hence a common neighbor among the vertices $x_i$. Thus $G$ is $4$-colorable. This completes the proof for $K_3$-free graphs.
\begin{figure}[!htbp]
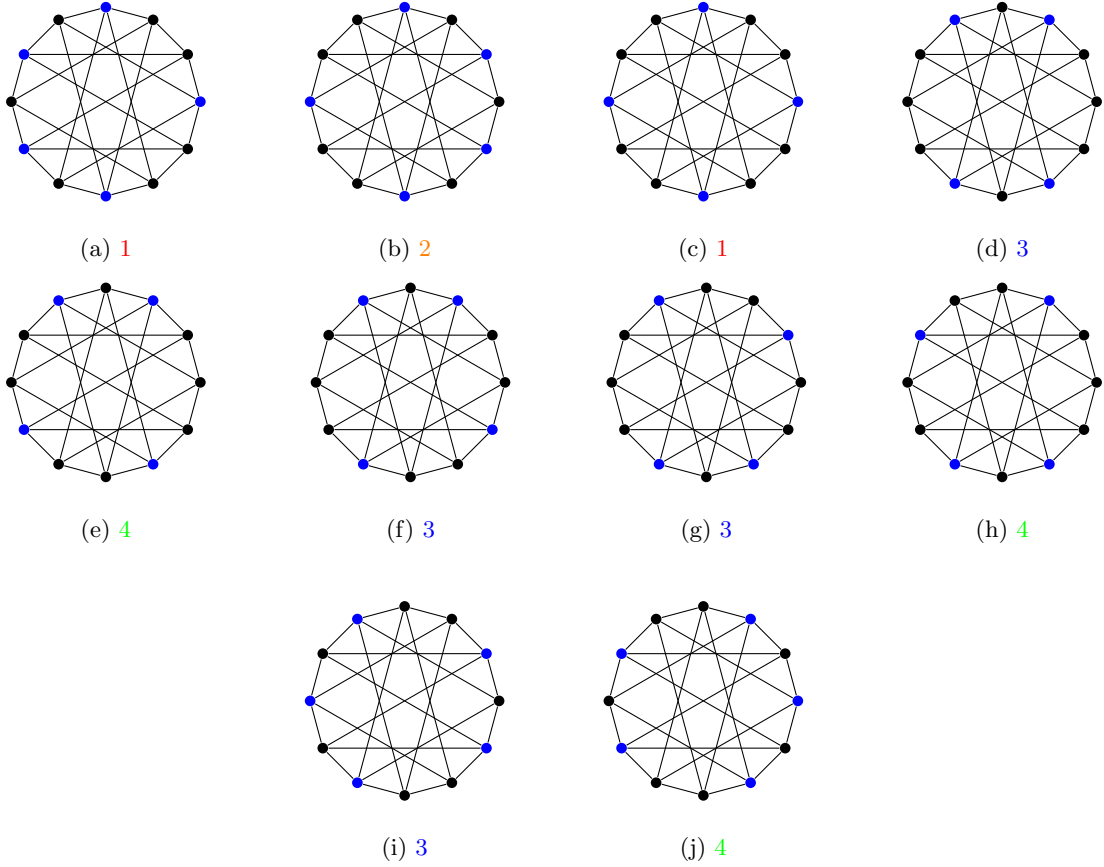

    \centering
    \begin{subfigure}[b]{0.222\textwidth}
        \centering

        \caption{\color{red}{1}}
        \label{fig:color of H4-a}
    \end{subfigure}
    \hfill 
    \begin{subfigure}[b]{0.222\textwidth}
        \centering
        %
        \caption{\color{orange}{2}}
        \label{fig:color of H4-b}
    \end{subfigure}
\hfill
\begin{subfigure}[b]{0.222\textwidth}
        \centering
        %
        \caption{\color{red}{1}}
        \label{fig:color of H4-c}
    \end{subfigure}
    \hfill 
    \begin{subfigure}[b]{0.222\textwidth}
        \centering
        %
        \caption{\color{blue}{3}}
        \label{fig:color of H4-d}
    \end{subfigure}
\hfill
\begin{subfigure}[b]{0.222\textwidth}
        \centering
        %
        \caption{\color{green}{4}}
        \label{fig:color of H4-e}
    \end{subfigure}
\hfill
    \vspace*{0.5cm}
    \begin{subfigure}[b]{0.222\textwidth}
        \centering
        %
        \caption{\color{blue}{3}}
        \label{fig:color of H4-f}
    \end{subfigure}
\hfill
\begin{subfigure}[b]{0.222\textwidth}
        \centering
        %
        \caption{\color{blue}{3}}
        \label{fig:color of H4-g}
    \end{subfigure}
    \hfill 
    \begin{subfigure}[b]{0.222\textwidth}
        \centering
        %
        \caption{\color{green}{4}}
        \label{fig:color of H4-h}
\end{subfigure}
\hfill\\
\begin{subfigure}[b]{0.222\textwidth}
    \mbox{}
\end{subfigure}
\hfill
\begin{subfigure}[b]{0.222\textwidth}
        \centering
        %
        \caption{\color{blue}{3}}
        \label{fig:color of H4-i}
    \end{subfigure}
    \hfill 
    \begin{subfigure}[b]{0.222\textwidth}
        \centering
        %
        \caption{\color{green}{4}}
        \label{fig:color of H4-j}
    \end{subfigure}
    \hfill
    \begin{subfigure}[b]{0.222\textwidth}
        \mbox{}
    \end{subfigure}

    \caption{The $4$-independent sets in $H_4$}
    \label{fig:color of H4}
\end{figure}

\subsection{Extension to $K_r$-free graphs}
The following theorem of Brandt \cite{2001-brandt} relates $K_{r-1}$-free graphs to maximal $K_{r}$-free graphs with large minimum degree; see also \cite{2011-lyle}. The \emph{join} of two vertex-disjoint graphs $G_1$ and $G_2$ is the graph $G$ with vertex set $V(G_1)\cup V(G_2)$ and edge set $E(G_1)\cup E(G_2)\cup\{uv:u\in V(G_1),\ v\in V(G_2)\}$.
\begin{theorem}[\cite{2001-brandt,2011-lyle}]\label{thm-decomposition Kr-free graph}
Let $r\geq4$ and $G$ be an $n$-vertex maximal $K_r$-free graph with minimum degree $\delta(G)\geq \frac{2r-5}{2r-3}n$. Then $G$ is the join of an independent set and a $K_{r-1}$-free graph.
\end{theorem}
\begin{proof}[Continuation of the proof of Theorem~\ref{thm-main thm}]
Suppose that $s\geq3$ and that every $K_s$-free graph $F$ with minimum degree at least $\frac{2s-5}{2s-3}|V(F)|$ is $(s+1)$-colorable. It suffices to prove that every $n$-vertex $K_{s+1}$-free graph with minimum degree at least $\frac{2s-3}{2s-1}n$ is $(s+2)$-colorable. After adding edges if necessary, let $G$ be a maximal $K_{s+1}$-free graph satisfying this degree condition. By Theorem~\ref{thm-decomposition Kr-free graph} with $r=s+1$, the graph $G$ is the join of an independent set $I$ and a $K_s$-free graph $T$. The degree condition gives $|T|\geq\delta(G)\geq \frac{2s-3}{2s-1}n$, and, for every $w\in T$,
\[
d_{T}(w)\geq \delta(G)-(n-|T|)\geq \frac{2s-5}{2s-3}|T|.
\]
By the induction hypothesis, $G[T]$ is $(s+1)$-colorable, and hence $G$ is $(s+2)$-colorable. The preceding subsection provides the base case $s=3$, completing the proof of Theorem~\ref{thm-main thm}.
\end{proof}

\section{Connections with subgraphs of Kneser graphs}\label{section-connection between kneser graph}
\subsection{Bounding the chromatic number}
In this section, we prove Theorems~\ref{thm-lower bound on chi of kneser subgraph imply k3 free graph} and~\ref{thm-sufficiency}. The first construction is motivated by Hajnal's construction.

The following theorem converts an $f(n)$-vertex subgraph of $KG(n,f(n))$ with large chromatic number into a triangle-free graph $G$ with minimum degree at least $\frac{1}{3}|V(G)|-\frac{2}{3}f(n)$ and equally large chromatic number.

\begin{theorem}\label{thm-lower bound on chi of kneser subgraph imply k3 free graph}
Let $f$ be a positive integer-valued function satisfying $f(n)=o(n)$. If $KG(n,f(n))$ has an $f(n)$-vertex subgraph $H$ with chromatic number at least $t$, then there is a $(3n+2f(n))$-vertex $K_3$-free graph $G$ with minimum degree at least $n=\frac{1}{3}|V(G)|-\frac{2}{3}f(n)$ and chromatic number at least $t$.
\end{theorem}

\begin{proof}
Let $f=f(n)$, and let $H$ be an $f$-vertex subgraph of $KG(n,f)$ with chromatic number at least $t$. We construct a $(3n+2f)$-vertex triangle-free graph $G$ with minimum degree at least $\frac{1}{3}|V(G)|-\frac{2}{3}f=n$ and chromatic number at least $t$. Partition $V(G)$ into $V(H)$ and independent sets $I_1,I_2$, where $|I_1|=2n+f$ and $|I_2|=n$. Include $E(H)$ and all edges between $I_1$ and $I_2$. Each vertex of $H$ corresponds to an $n$-subset of $[2n+f]$. Denote $I_1=\{x_1,\dots,x_{2n+f}\}$, and join $x_i$ to $w\in V(H)$ if and only if $i\in w$. The resulting graph is triangle-free and has order $3n+2f$. Every vertex of $I_1\cup V(H)$ has degree at least $n$, and every vertex of $I_2$ has degree $2n+f$. Thus $\delta(G)\geq n=|V(G)|/3-2f/3$, while $\chi(G)\geq\chi(H)\geq t$.
\end{proof}
This construction is a direct variant of Hajnal's construction.
In the latter, $KG(n,k)$ is a $\binom{2n+k}{n}$-vertex subgraph of $KG(\frac{2n\ell}{2n+k},\frac{2k\ell}{2n+k})$.
Theorem~\ref{thm-24000} shows that the large bipartite subgraph in this construction is necessary.
We next give a sufficient condition, in terms of Kneser graphs, for the chromatic number to be bounded.
\begin{theorem}\label{thm-sufficiency}
Let $f$ be a positive integer-valued function satisfying $f(n)=o(n)$, let $t\geq1$ be a fixed integer, and assume that $n$ is sufficiently large. Set $P_n=\lceil n/3-10f(n)\rceil$. If every $f(n)$-vertex subgraph of $KG(P_n,21f(n))$ is $t$-colorable, then every $n$-vertex $K_3$-free graph with minimum degree at least $\frac{1}{3}n-f(n)$ is $(t+10^{391})$-colorable.
\end{theorem}

\begin{proof}
Let $G$ be an $n$-vertex $K_3$-free graph with minimum degree at least $\frac{n}{3}-f$, where we write $f=f(n)$ and assume that $n$ is sufficiently large. By adding edges if necessary, we may assume that $G$ is maximal triangle-free. If $G$ satisfies $\D_4$, then Theorem~\ref{thm-D3 and D4 implies blow-up of Vega graph} implies that $G$ is a blow-up of either an Andr\'{a}sfai graph $\Gamma_i$ or a Vega graph, and hence is $4$-colorable. We may therefore assume that $G$ violates $\D_4$. Thus there are an integer $k\in\{1,2,3,4\}$ and a $3k$-element multiset $H=\{x_1,\dots,x_{3k}\}\subseteq V(G)$ such that every vertex of $G$ is adjacent to at most $k$ elements of $H$, counted with multiplicity.
We have $d(x_1)+\dots+d(x_{3k})\geq kn-3kf$, and hence $|N(x_1)\cup\dots\cup N(x_{3k})|\geq n-3f$. Choose an $(n-3f)$-subset $D\subseteq N(x_1)\cup\dots\cup N(x_{3k})$, and let $R=V(G)\setminus D$. The set $D$ can be colored with at most $3k\leq12$ colors, so it suffices to color $R$ with $t+(10^{391}-12)$ colors. For every $w\in V(G)$, we have $d_D(w)\geq\delta(G)-|R|\geq\frac{1}{3}n-4f$. Choose an edge $uv\in E(G)$ and disjoint independent sets $I_1\subseteq N(u)\cap D$ and $I_2\subseteq N(v)\cap D$ with $|I_1|=|I_2|=\frac{1}{3}n-4f$. Let $C=D\setminus(I_1\cup I_2)$. Then $|C|=n-3f-2(\frac{n}{3}-4f)=\frac{n}{3}+5f$. We say that $A\subseteq V(G)$ \emph{controls} a vertex $w$ if $d_A(w)>|A|/2$. The vertices controlled by a fixed set $A$ form an independent set and may therefore receive one color.


\begin{claim}\label{claim-adjacent to half of I1 I2 C}
At the cost of at most $3$ colors, we may assume that every vertex in $R$ has at most $\frac{1}{2}|I_1|$, $\frac{1}{2}|I_2|$, and $\frac{1}{2}|C|$ neighbors in $I_1$, $I_2$, and $C$, respectively.\footnote{The phrase ``at the cost of at most $x$ colors'' means that we remove at most $x$ disjoint independent sets and assign them distinct colors.}
\end{claim}
Indeed, every vertex violating Claim~\ref{claim-adjacent to half of I1 I2 C} is controlled by one of $I_1,I_2,C$. These vertices are therefore $3$-colorable, and the remaining vertices satisfy the claim.
We next separate the vertices in $R$ according to whether their degrees into $I_1$ and $I_2$ lie in a fixed middle interval.
Set $\varepsilon=1/441$ and
\[
q_\varepsilon:=\left\lceil\frac{\log_2(1/\varepsilon)}{3\varepsilon}\right\rceil=1292.
\]
Define three subsets $X_1,X_2,Y$ of $R$ by, for $i\in\{1,2\}$,
\[
X_i =\{w\in R: \varepsilon n\leq d_{I_i}(w)\leq (\frac{1}{6}-\varepsilon)n \},
\]
and
\[
Y=\{w\in R: d_{I_i}(w)<\varepsilon n\quad  \text{or}\quad (\frac{1}{6}-\varepsilon)n< d_{I_i}(w)\leq \frac{1}{6}n-2f=\frac{1}{2}|I_i|\quad \text{for each } i=1,2 \}.
\]
Note that $R=X_1 \cup X_2 \cup Y$ and $Y\cap (X_1\cup X_2)=\emptyset$. 
The following claim gives a uniform chromatic bound for vertices having many neighbors both in a $(\frac{1}{3}n-4f)$-independent set $I$ and in $D\setminus I$.

\begin{claim}\label{lem-color the vertex with mid neighbors in I}
Let $I\subseteq D$ be an independent set of order $\frac{1}{3}n-4f$. Then the set of vertices in $R$ having at least $\varepsilon n$ neighbors in $I$ and at least $(\frac{1}{6}+\frac{1}{2}\varepsilon)n$ neighbors in $D\setminus I$ is colorable with at most $2^{q_\varepsilon}$ colors.
\end{claim}
\begin{proof}
Let $R_1$ denote the set in the statement. Choose $x_1\in R_1$ and color all vertices of $R_1$ controlled by $N(x_1)\cap(D\setminus I)$. Let $R_2$ be the set of uncolored vertices. Every $w_2\in R_2$ has at most $\frac{|N(x_1)\cap(D\setminus I)|}{2}$ neighbors in $N(x_1)\cap(D\setminus I)$, and
\begin{align*}
|(N(x_1)\cup N(w_2))\cap (D\setminus I)|&\geq |N(x_1)\cap (D\setminus I)|+(\frac{1}{6}+\frac{1}{2}\e)n-|N(x_1)\cap N(w_2)\cap (D\setminus I)|\\
&\geq \frac{1}{2}|N(x_1)\cap (D\setminus I)|+(\frac{1}{6}+\frac{1}{2}\e)n\\
&\geq (1+\frac{1}{2})(\frac{1}{6}+\frac{1}{2}\e)n.
\end{align*}
Suppose that $x_1,\dots,x_{i-1}$ and $R_i$ have been defined. If $R_i=\emptyset$, stop. Otherwise, choose $x_i\in R_i$, color all vertices of $R_i$ controlled by some vertex set $(N(x_{i_1})\cup\dots\cup N(x_{i_t}))\cap(D\setminus I)$, where $t\in\{1,\dots,i\}$, $i_k$'s are distinct, $i_t=i$ and $i_{k}\in \{1,\dots,i\}$ for each $k\in\{1,\dots,t\}$. These vertices are controlled by one of at most $2^{i-1}$ vertex sets and thus they are the union of at most $2^{i-1}$ independent sets. By induction, up to the $i$-th step, there are totally $2^i$ independent sets being colored.
Let $R_{i+1}$ be the set of uncolored vertices. We consider the subsequence $\{x_{i_1},\dots,x_{i_t}\}\subseteq \{x_1,\dots,x_{i}\}$, where $1\leq t\leq i$ and $i_t$'s are distinct. When $t=1$, it is trivial since
\[
|N(x_{i_1})\cap (D\setminus I)|\geq (\frac{1}{6}+\frac{1}{2}\e)n.
\]
Assume $t\geq2$. Inductively, we have
\[
|(N(x_{i_1})\cup\dots \cup N(x_{i_{t-1}}))\cap (D\setminus I)|\geq (1+\frac{1}{2}+\dots+\frac{1}{2^{t-2}})(\frac{1}{6}+\frac{1}{2}\e)n.
\]
Note that any vertex in $R_{i}$ is not controlled by $N(x_{i_1})\cup\dots \cup N(x_{i_{t-1}})$, and hence we have
\begin{align*}
&\quad |(N(x_{i_{1}})\cup\dots\cup N(x_{i_t}))\cap (D\setminus I)|\\
&\geq |(N(x_{i_1})\cup \dots\cup N(x_{i_{t-1}}))\cap (D\setminus I)|+(\frac{1}{6}+\frac{1}{2}\e)n-|(N(x_{i_1})\cup\dots \cup N(x_{i_{t-1}}))\cap N(x_{i_{t}})\cap (D\setminus I)|\\
&\geq \frac{1}{2}|(N(x_{i_1})\cup \dots\cup N(x_{i_{t-1}}))\cap (D\setminus I)|+(\frac{1}{6}+\frac{1}{2}\e)n\\
&\geq (1+\frac{1}{2}+\dots+\frac{1}{2^{t-1}})(\frac{1}{6}+\frac{1}{2}\e)n.
\end{align*}
Suppose that the process runs for $j$ steps and produces $x_1,\dots,x_j$. It suffices to bound $j$. Since each $x_i$ has at least $\varepsilon n$ neighbors in $I$, averaging gives a vertex $x\in I$ adjacent to at least $\frac{\varepsilon n}{\frac{n}{3}-4f}j\geq3\varepsilon j$ of the vertices $x_1,\dots,x_j$. The union of the neighborhoods in $D\setminus I$ of these vertices has order at least
\[
(\frac{1}{6}+\frac{1}{2}\varepsilon)n (1+\frac{1}{2}+\frac{1}{4}+\dots+\frac{1}{2^{3\varepsilon j-1}})=(\frac{1}{3}+\varepsilon)(1-\frac{1}{2^{3\varepsilon j}})n.
\]
If $1-\frac{1}{2^{3\varepsilon j}}\geq1-\varepsilon$, then this union has order at least $(\frac{1}{3}+\varepsilon)(1-\varepsilon)n\geq(\frac{1}{3}+\frac{1}{2}\varepsilon)n$. Because $G$ is triangle-free, $x$ is adjacent to no vertex of this union or of $I$. These two sets have total order at least $(\frac{2}{3}+\frac{1}{2}\varepsilon)n-4f$, so $d(x)\leq(\frac{1}{3}-\frac{1}{2}\varepsilon)n+4f$, contrary to $\delta(G)\geq\frac{1}{3}n-f$ for sufficiently large $n$. Hence $1-2^{-3\varepsilon j}<1-\varepsilon$, which gives $j<\frac{\log_2(1/\varepsilon)}{3\varepsilon}\leq q_\varepsilon$. 
Recall that until the $i$-th step we totally color at most $2^i$ independent sets.
Thus $R_1$ can be colored with at most $2^{q_{\varepsilon}}$ colors.
\end{proof}

Every vertex in $X_1$ has at least $\varepsilon n$ neighbors in $I_1$ and at least $(\frac{1}{6}+\varepsilon)n-4f\geq (\frac{1}{6}+\frac{1}{2}\varepsilon)n$ neighbors in $D\setminus I_1$.
By Claim~\ref{lem-color the vertex with mid neighbors in I}, each of $X_1$ and $X_2$ can be colored with at most $2^{q_\varepsilon}$ colors. It remains to color $Y$. The following claim shows that a vertex of $Y$ with many neighbors in both $I_1$ and $I_2$ has few neighbors in $C$.
\begin{claim}\label{claim-neighbor of vertex in Y is regular}
At the cost of at most $2^{q_\varepsilon}$ colors, we may assume that every $w\in Y$ satisfying $d_{I_1}(w)\geq(\frac{1}{6}-\varepsilon)n$ and $d_{I_2}(w)\geq(\frac{1}{6}-\varepsilon)n$ also satisfies $d_C(w)\leq2\varepsilon n$.
\end{claim}
\begin{proof}
Let $Z$ be the set of vertices $w\in Y$ satisfying $d_{I_1}(w)\geq(\frac{1}{6}-\varepsilon)n$, $d_{I_2}(w)\geq(\frac{1}{6}-\varepsilon)n$, and $d_C(w)\geq2\varepsilon n$. For each $w\in Z$, we have $d_{I_2\cup C}(w)\geq(\frac{1}{6}+\varepsilon)n$ and $d_{I_1}(w)\geq(\frac{1}{6}-\varepsilon)n>\varepsilon n$. Claim~\ref{lem-color the vertex with mid neighbors in I} therefore shows that $Z$ can be colored with at most $2^{q_\varepsilon}$ colors.
\end{proof}
For a vertex $w\in Y$, if $(\frac{1}{6}-\e)n\leq d_{I_1}(w)\leq \frac{1}{6}n-2f$ and $d_{I_2}(w)<\e n$, then
\[
d_{C}(w)\geq d_{D}(w)-d_{I_1}(w)-d_{I_2}(w)\geq \frac{1}{3}n-4f-(\frac{1}{6}n-2f)-\e n=(\frac{1}{6}-\e)n-2f.
\]
Together with Claim~\ref{claim-adjacent to half of I1 I2 C}, this gives $(\frac{1}{6}-\e)n-2f\leq d_C(w)\leq\frac{1}{2}|C|=\frac{1}{6}n+\frac{5}{2}f$.
If $Y$ is independent, it requires only one color. Otherwise, choose an edge $uv\in E(Y)$. By Claims~\ref{claim-adjacent to half of I1 I2 C} and~\ref{claim-neighbor of vertex in Y is regular}, and by the definition of $Y$, the disjoint neighborhoods $N(u)$ and $N(v)$ lead to four candidate patterns, up to interchanging $u,v$ and $I_1,I_2$. Choose disjoint independent sets $I_1'\subseteq N(u)\cap D$ and $I_2'\subseteq N(v)\cap D$, each of order $\frac{n}{3}-4f$; see Figure~\ref{fig:four-neighborhood-patterns}. To record the intended quantitative ranges, set
\[
\mathcal S=[0,2\varepsilon n],\qquad
\mathcal M=\left[\left(\frac16-\varepsilon\right)n,\frac16n-2f\right],\qquad
\mathcal M_C=\left[\left(\frac16-\varepsilon\right)n-2f,\frac16n+\frac52f\right].
\]
The six entries in each row below give, in order, the sizes of the intersections of $I_1'$ and $I_2'$ with $I_1,I_2,C$:
\[
\begin{array}{c|ccc|ccc}
 & |I_1'\cap I_1|&|I_1'\cap I_2|&|I_1'\cap C|
 & |I_2'\cap I_1|&|I_2'\cap I_2|&|I_2'\cap C|\\ \hline
1&\mathcal M&\mathcal M&\mathcal S&\mathcal M&\mathcal M&\mathcal S\\
2&\mathcal M&\mathcal M&\mathcal S&\mathcal S&\mathcal M&\mathcal M_C\\
3&\mathcal S&\mathcal M&\mathcal M_C&\mathcal S&\mathcal M&\mathcal M_C\\
4&\mathcal M&\mathcal S&\mathcal M_C&\mathcal S&\mathcal M&\mathcal M_C
\end{array}
\]

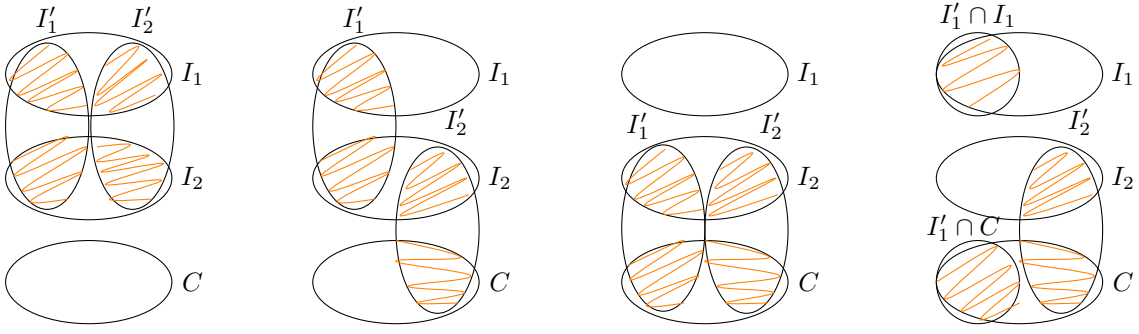
\begin{figure}
    \centering
    \begin{subfigure}[b]{0.2\textwidth}
        \centering
        \begin{tikzpicture}[scale=0.55]

\draw  (-1,1.5) ellipse (2 and 1);
\draw  (-1,-1) ellipse (2 and 1);
\draw  (-1,-3.5) ellipse (2 and 1);
\node at (1.5,1.5) {$I_1$};
\node at (1.5,-1) {$I_2$};
\node at (1.5,-3.5) {$C$};
\draw  (-2,0.25) ellipse (1 and 2) ;
\draw  (0.05,0.25) ellipse (1 and 2);

\node at (-2,2.8) {$I_1'$};
\node at (0.3,2.8) {$I_2'$};

\draw [orange] plot[smooth, tension=.7] coordinates {(-1.95,2.2) (-2.9,1.35) (-1.5,2) (-2.7,1) (-1.25,1.6) (-2.45,0.85) (-1.15,1.2) (-2,0.65) (-1.05,0.75)};
\draw [orange] plot[smooth, tension=.7] coordinates {(-2.85,-0.6) (-1.5,0) (-2.8,-0.95) (-1.05,-0.25) (-2.6,-1.3) (-1.15,-0.85) (-2.4,-1.55) (-1.5,-1.5)};
\draw [orange] plot[smooth, tension=.7] coordinates {(-0.6,1.7) (0.2,2.15) (-0.8,1) (0.4,1.85) (-0.85,0.75) (0.75,1.5) (-0.5,0.6) (0.6,1)};
\draw [orange] plot[smooth, tension=.7] coordinates {(-0.8,-0.25) (0,-0.2) (-0.8,-0.6) (0.45,-0.45) (-0.7,-0.85) (0.8,-0.65) (-0.6,-1.15) (0.75,-1.05) (-0.5,-1.5) (0.5,-1.5)};

        \end{tikzpicture}
    \end{subfigure}
    \hfill 
    \begin{subfigure}[b]{0.2\textwidth}
        \centering
        \begin{tikzpicture}[scale=0.55]

\draw  (-1,1.5) ellipse (2 and 1);
\draw  (-1,-1) ellipse (2 and 1);
\draw  (-1,-3.5) ellipse (2 and 1);
\node at (1.5,1.5) {$I_1$};
\node at (1.5,-1) {$I_2$};
\node at (1.5,-3.5) {$C$};
\draw  (-2,0.25) ellipse (1 and 2) ;
\draw  (0,-2.25) ellipse (1 and 2);

\node at (-2,2.8) {$I_1'$};
\node at (0.45,0.35) {$I_2'$};

\draw [orange] plot[smooth, tension=.7] coordinates {(-1.95,2.2) (-2.9,1.35) (-1.5,2) (-2.7,1) (-1.25,1.6) (-2.45,0.85) (-1.15,1.2) (-2,0.65) (-1.05,0.75)};
\draw [orange] plot[smooth, tension=.7] coordinates {(-2.85,-0.6) (-1.5,0) (-2.8,-0.95) (-1.05,-0.25) (-2.6,-1.3) (-1.15,-0.85) (-2.4,-1.55) (-1.5,-1.5)};

\draw  [orange] plot[smooth, tension=.7] coordinates {(-0.65,-0.8) (0.15,-0.4) (-0.75,-1.25) (0.55,-0.65) (-0.9,-1.65) (0.7,-1) (-0.9,-1.9) (0.75,-1.4)};
\draw [orange] plot[smooth, tension=.7] coordinates {(-1,-2.5) (0.55,-2.9) (-1,-3) (0.8,-3.35) (-0.7,-3.6) (0.65,-3.75) (-0.5,-4) (0.5,-4)};

        \end{tikzpicture}
    \end{subfigure}
\hfill 
    \begin{subfigure}[b]{0.2\textwidth}
        \centering
        \begin{tikzpicture}[scale=0.55]

\draw  (-1,1.5) ellipse (2 and 1);
\draw  (-1,-1) ellipse (2 and 1);
\draw  (-1,-3.5) ellipse (2 and 1);
\node at (1.5,1.5) {$I_1$};
\node at (1.5,-1) {$I_2$};
\node at (1.5,-3.5) {$C$};
\draw  (-2,-2.2) ellipse (1 and 2) ;
\draw  (0,-2.25) ellipse (1 and 2);
\node at (-2.60,0.25) {$I_1'$};
\node at (0.60,0.25) {$I_2'$};

\draw [orange] plot[smooth, tension=.7] coordinates {(-1.95,-0.3) (-2.9,-1.15) (-1.5,-0.5) (-2.7,-1.5) (-1.25,-0.9) (-2.45,-1.65) (-1.15,-1.3) (-2,-1.85) (-1.05,-1.75)};
\draw [orange] plot[smooth, tension=.7] coordinates {(-2.85,-3.1) (-1.5,-2.5) (-2.8,-3.45) (-1.05,-2.75) (-2.6,-3.8) (-1.15,-3.35) (-2.4,-4.05) (-1.5,-4)};

\draw  [orange] plot[smooth, tension=.7] coordinates {(-0.65,-0.8) (0.15,-0.4) (-0.75,-1.25) (0.55,-0.65) (-0.9,-1.65) (0.7,-1) (-0.9,-1.9) (0.75,-1.4)};
\draw [orange] plot[smooth, tension=.7] coordinates {(-1,-2.5) (0.55,-2.9) (-1,-3) (0.8,-3.35) (-0.7,-3.6) (0.65,-3.75) (-0.5,-4) (0.5,-4)};

        \end{tikzpicture}
    \end{subfigure}
\hfill
\begin{subfigure}[b]{0.2\textwidth}
\centering
\begin{tikzpicture}[scale=0.55]
\draw  (-1,1.5) ellipse (2 and 1);
\draw  (-1,-1) ellipse (2 and 1);
\draw  (-1,-3.5) ellipse (2 and 1);
\node at (1.5,1.5) {$I_1$};
\node at (1.5,-1) {$I_2$};
\node at (1.5,-3.5) {$C$};
\draw  (0,-2.25) ellipse (1 and 2);
\node at (-2,2.9) {$I_1'\cap I_1$};
\node at (-2.35,-2.20) {$I_1'\cap C$};
\node at (0.45,0.35) {$I_2'$};

\draw  [orange] plot[smooth, tension=.7] coordinates {(-0.65,-0.8) (0.15,-0.4) (-0.75,-1.25) (0.55,-0.65) (-0.9,-1.65) (0.7,-1) (-0.9,-1.9) (0.75,-1.4)};
\draw [orange] plot[smooth, tension=.7] coordinates {(-1,-2.5) (0.55,-2.9) (-1,-3) (0.8,-3.35) (-0.7,-3.6) (0.65,-3.75) (-0.5,-4) (0.5,-4)};
\draw  (-2,1.5) ellipse (1 and 1);
\draw  (-2,-3.5) ellipse (1 and 1);
\draw [orange] plot[smooth, tension=.7] coordinates {(-3,1.5)};
\draw [orange] plot[smooth, tension=.7] coordinates {(-1.7,2.35) (-2.85,1.7) (-1.3,2.15) (-2.85,1.1) (-1,1.6) (-2.25,0.7)};
\draw [orange] plot[smooth, tension=.7] coordinates {(-3,-3.5) (-1.55,-2.65) (-2.8,-3.9) (-1.2,-3.1) (-2.4,-4.2) (-1,-3.5) (-1.95,-4.4) (-1,-4)};

\end{tikzpicture}
\end{subfigure}
    \caption{Four possible neighborhood patterns for an edge in $Y$}
    \label{fig:four-neighborhood-patterns}
\end{figure}
The following claim treats an independent set whose intersections with both $I_1$ and $I_2$ lie between $(\frac16-\varepsilon)n$ and $\frac16n-2f$. It shows that every disjoint independent set of the same order meets $C$ in at most $6\varepsilon n$ or in at least $(\frac13-10\varepsilon)n$. In particular, it excludes the second pattern in Figure~\ref{fig:four-neighborhood-patterns}.
\begin{claim}\label{lem-neighbor in C is extremal}
Assume that $I_1'\subseteq D$ is an independent set of order $\frac{1}{3}n-4f$ such that $(\frac{1}{6}-\varepsilon)n\leq |I_1\cap I_1'|\leq \frac{1}{6}n-2f$ and $(\frac{1}{6}-\varepsilon)n\leq |I_2\cap I_1'|\leq \frac{1}{6}n-2f$. Then we may assume that every independent set $I_2'\subseteq D\setminus I_1'$ of order $\frac{1}{3}n-4f$ satisfies either $|I_2'\cap C|\leq 6\varepsilon n$ or $|I_2'\cap C|\geq (\frac{1}{3}-10\varepsilon)n$.
\end{claim}
\begin{proof}
Suppose otherwise that $6\varepsilon n<|I_2'\cap C|<(\frac{1}{3}-10\varepsilon)n$. It suffices to show that six subsets of $V(G)$ control all vertices of $Y$.
The preceding analysis of $(I_1,I_2)$ remains valid with $(I_1',I_2')$ in its place. Including all deletions made before this claim, the cumulative cost is at most $3+9\cdot2^{q_\varepsilon}$ colors; after these deletions, for every $x\in Y$ and $i\in\{1,2\}$, either $d_{I_i'}(x)<\varepsilon n$ or $(\frac{1}{6}-\varepsilon)n\leq d_{I_i'}(x)\leq \frac{1}{6}n-2f$.
Take an arbitrary vertex $w\in Y$.

If $d_{I_1}(w)\geq \varepsilon n$, then it follows from the definition of $Y$ that $(\frac{1}{6}-\varepsilon) n\leq d_{I_1}(w)\leq \frac{1}{6}n-2f $. By Claim \ref{claim-neighbor of vertex in Y is regular} and $d_{D}(w)\geq \frac{1}{3}n-4f$, we have either $d_{I_2}(w)< \varepsilon n$ or $d_{C}(w)\leq 2\varepsilon n$.
If $d_{I_2}(w)< \varepsilon n$ and $d_{I_1 \cap I_1'}(w)\geq \e n$, then $d_{I_1'}(w)\geq (\frac{1}{6}-\e)n$ and $d_{I_1 \cap I_1'}(w)\geq d_{I_1'}(w)-d_{I_2}(w)-(\frac{1}{3}n-4f-|I_1\cap I_1'|-|I_2\cap I_1'|) \geq (\frac{1}{6}-4\e)n$, and thus $d_{I_1\setminus I_1'}(w)\leq 4\e n$. Altogether, $d_{(I_1\cap I_1')\cup C}(w)\geq (\frac{1}{3}-6\e)n$ and $|(I_1\cap I_1')\cup C|\leq (\frac{1}{2}+\e)n$, so $(I_1\cap I_1')\cup C$ controls $w$.
If $d_{I_2}(w)< \varepsilon n$ and $d_{I_1 \cap I_1'}(w)< \e n$, then $d_{(I_1\setminus I_1')\cup C}(w)\geq (\frac{1}{3}-3\e)n$ and $|(I_1\setminus I_1')\cup C|\leq (\frac{1}{2}+2\e)n$, and thus $(I_1\setminus I_1')\cup C$ controls $w$. 
If $d_{C}(w)\leq 2\varepsilon n$ and $d_{I_2'}(w)\geq \e n$, then $d_{I_2'}(w)\geq (\frac{1}{6}-\e)n$, $d_{(I_1\cup I_2)\cap I_2'}(w)\geq d_{I_{2}'}(w)-d_{C}(w) \geq(\frac{1}{6}-3\e)n$, and $|(I_1\cup I_2)\cap I_2'|=|I_2'|-|I_2'\cap C|\leq(\frac{1}{3}n-4f)-6\e n$, which implies $(I_1\cup I_2)\cap I_2'$ controls $w$.
If $d_{C}(w)\leq 2\varepsilon n$ and $d_{I_2'}(w)< \e n$, then $d_{(I_1\cup I_2)\setminus I_2'}(w)\geq (\frac{1}{3}-4\e)n$ and $|(I_1\cup I_2)\setminus I_2'|=|I_1\cup I_2|-|(I_1\cup I_2)\cap I_2'|=|I_1\cup I_2|-|I_2'|+|C\cap I_2'| \leq \frac{1}{3}n-4f+(\frac{1}{3}-10\e)n$, which implies that $(I_1\cup I_2)\setminus I_2'$ controls $w$.

If $d_{I_1}(w)< \varepsilon n$ and $d_{I_2\cap I_1'}(w)< \e n$, then $d_{(I_2\setminus I_1')\cup C}(w)\geq (\frac{1}{3}-3 \e)n$ and $|(I_2\setminus I_1')\cup C|\leq (\frac{1}{2}+2\e)n$, which implies that $(I_2\setminus I_1')\cup C$ controls $w$.
If $d_{I_1}(w)< \varepsilon n$ and $d_{I_2\cap I_1'}(w)\geq \e n$, then $d_{I_1'}(w)\geq (\frac{1}{6}-\e) n$, $d_{I_2 \cap I_1'}(w)\geq (\frac{1}{6}-4\e) n$, $d_{I_2 \setminus I_1'}(w)\leq 4\e n$. Thus $d_{(I_2\cap I_1')\cup C}(w)\geq (\frac{1}{3}-6\e)n$ and $|(I_2\cap I_1')\cup C|\leq (\frac{1}{2}+\e)n$, which implies $(I_2\cap I_1')\cup C$ controls $w$.

Thus every vertex of $Y$ is controlled by one of $(I_1\cap I_1')\cup C$, $(I_1\setminus I_1')\cup C$, $(I_1\cup I_2)\cap I_2'$, $(I_1\cup I_2)\setminus I_2'$, $(I_2\setminus I_1')\cup C$, and $(I_2\cap I_1')\cup C$. Hence $Y$ is $6$-colorable.
\end{proof}
We may consider each component separately and hence assume that $Y$ is connected.
Claims~\ref{claim-adjacent to half of I1 I2 C} and~\ref{lem-neighbor in C is extremal} ensure that, if some $u\in Y$ satisfies $(\frac{1}{6}-\e)n\leq d_{I_1}(u)\leq \frac{1}{6}n-2f$ and $(\frac{1}{6}-\e)n\leq d_{I_2}(u)\leq \frac{1}{6}n-2f$, then every neighbor $v$ of $u$ in $Y$ satisfies the same inequalities. Since $Y$ is connected, the property then holds throughout $Y$.

The next claim is the analogue needed when one of the independent parts is replaced by $C$.
Notice that the proof of Claim \ref{lem-neighbor in C is extremal} does not depend on whether $I_1,I_2,C$ are independent, but only depend on the size of them and how these sets intersect with the neighborhood of vertex in $Y$. For example, the estimate of $N(w)\cap I_1$ and $N(w)\cap C$ are very similar (either at most $2\e n$ or between $(\frac{1}{6}-\e)n-2f$ and $\frac{1}{6}n+\frac{5}{2}f$), and the size of $I_1$ and $C$ have a small difference $9f$. Hence the proof of Claim \ref{claim-exclude-fourth-neighborhood-pattern} is similar and we omit it.
\begin{claim}\label{claim-exclude-fourth-neighborhood-pattern}
Assume that $I_1'\subseteq D$ is an independent set of order $\frac{1}{3}n-4f$ such that $(\frac{1}{6}-\varepsilon)n\leq |I_1\cap I_1'|\leq \frac{1}{6}n-2f$ and $(\frac{1}{6}-\varepsilon)n-2f\leq |C\cap I_1'|\leq \frac{1}{6}n+\frac{5}{2}f$. Then we may assume that every independent set $I_2'\subseteq D\setminus I_1'$ of order $\frac{1}{3}n-4f$ satisfies either $|I_2'\cap I_2|\leq 6\varepsilon n$ or $|I_2'\cap I_2|\geq (\frac{1}{3}-10\varepsilon)n$.
\end{claim}
This claim excludes the fourth pattern in Figure~\ref{fig:four-neighborhood-patterns}.
Similarly, if some $u\in Y$ satisfies $(\frac{1}{6}-\e)n\leq d_{I_1}(u)\leq \frac{1}{6}n-2f$ and $(\frac{1}{6}-\e)n-2f\leq d_C(u)\leq \frac{1}{6}n+\frac{5}{2}f$, then every vertex $v\in Y$ satisfies the same inequalities. Thus the neighborhood of each vertex in $Y$ is concentrated in one of $I_1\cup I_2$, $I_1\cup C$, and $I_2\cup C$, up to the error bounds displayed above. In either of the last two cases, choose an edge $uv\in E(Y)$ and disjoint independent sets $I_1'\subseteq D\cap N(u)$ and $I_2'\subseteq D\cap N(v)$, each of order $\frac{1}{3}n-4f$. Replacing $(I_1,I_2)$ by $(I_1',I_2')$ shows that the neighborhood of every vertex in $Y$ is concentrated in $I_1'\cup I_2'$ with the same bounds. Repeating the preceding argument for this new pair costs at most another $9+9\cdot 2^{q_\varepsilon}$ colors. We may therefore assume that the neighborhood of every vertex in $Y$ is concentrated in $I_1\cup I_2$ with the displayed error bounds.
Choose an edge $uv\in E(Y)$ and disjoint independent sets $I_1'\subseteq D\cap N(u)$ and $I_2'\subseteq D\cap N(v)$, each of order $\frac{1}{3}n-4f$. Then, for all $i,j\in\{1,2\}$, $(\frac{1}{6}-\e)n\leq|I_i\cap I_j'|\leq \frac{1}{6}n-2f$ and $|C\cap I_j'|\leq2\e n$.
The next claim restricts the distribution of the neighbors of vertices in $I_{1}\cup I_{2}$.
\begin{claim}\label{claim-vertex in I1I2 is critical for C}
For every $w\in I_1\cup I_2$, either $d_C(w)\leq15\e n$ or $d_C(w)\geq(\frac{1}{3}-10\e)n$.
\end{claim}
\begin{proof}
Let $w\in (I_1\cup I_2)\cap (I_1'\cup I_2')$, and choose $a,b\in\{1,2\}$ with $w\in I_a\cap I_b'$. Take a set $B\subseteq N_D(w)$ of order $\frac{1}{3}n-4f$. Since $B$ is disjoint from both $I_a$ and $I_b'$, writing $a'=3-a$ gives
\[
|B\cap C|\geq |B|-|I_{a'}\setminus I_b'|
\geq \left(\frac{1}{3}n-4f\right)-\left(\frac{1}{3}n-4f-\left(\frac{1}{6}-\varepsilon\right)n\right)
=\left(\frac{1}{6}-\varepsilon\right)n>6\varepsilon n.
\]
Apply Claim~\ref{lem-neighbor in C is extremal} with $I_b'$ as its first independent set and $B$ as its disjoint second independent set. The low alternative is excluded by the preceding inequality, and hence $d_C(w)\geq|B\cap C|\geq(\frac{1}{3}-10\varepsilon)n$.

For vertex $w\in (I_1\cup I_2)\setminus (I_1'\cup I_2')$, without loss of generality we assume $w\in I_2$. If $15\e n < d_{C}(w)< (\frac{1}{3}-10\e)n$, then $d_{I_1}(w)\geq \frac{1}{3}n-4f-d_{C}(w)\geq 9\e n$ and
\[
d_{I_1\cap(I_1'\cup I_2')}(w)\geq d_{I_1}(w)+|I_1\cap (I_1'\cup I_2')|-|I_1|\geq 9\e n+2(\frac{1}{6}-\e)n-(\frac{1}{3}n-4f)\geq 7\e n.
\]
Choose $x\in N(w)\cap I_1\cap(I_1'\cup I_2')$. Since $d_C(w)>15\e n$ and $|C|=\frac{n}{3}+5f$, triangle-freeness gives $d_C(x)<|C|-15\e n\leq\frac{1}{3}n-10\e n$, contradicting $d_C(x)\geq(\frac{1}{3}-10\e)n$.
\end{proof}
Vertices in $C$ satisfy a similar dichotomy.
\begin{claim}\label{claim-vertex in C is critical for C}
For each vertex $w\in C$, we can assume either $d_{C}(w)\leq 17\e n$ or $d_{C}(w)\geq(\frac{1}{3}-5\e)n$.
\end{claim}
\begin{proof}
Suppose otherwise that some $w_0\in C$ satisfies $17\e n<d_C(w_0)<(\frac{1}{3}-5\e)n$. Then $d_{I_1\cup I_2}(w_0)\geq \frac{1}{3}n-4f-d_C(w_0)\geq5\e n-4f$, which implies that
\[
\begin{aligned}
d_{(I_1\cup I_2)\cap(I_1'\cup I_2')}(w_0)
&\geq d_{I_1\cup I_2}(w_0)+ |(I_1\cup I_2)\cap(I_1'\cup I_2')|-|I_1\cup I_2|\\
&\geq 5\e n-4f+4(\frac{1}{6}-\e)n-2(\frac{1}{3}n-4f)\\
&\geq \e n+4f.
\end{aligned}
\]
Choose a vertex $x$ in the nonempty set $N(w_0)\cap(I_1\cup I_2)\cap(I_1'\cup I_2')$. Triangle-freeness gives $d_C(x)\leq |C|-d_{C}(w_0)\leq\frac{1}{3}n+5f-17\e n$. The first paragraph of the proof of Claim~\ref{claim-vertex in I1I2 is critical for C} gives $d_C(x)\geq(\frac{1}{3}-10\e)n$, a contradiction because $5f<7\varepsilon n$ for all sufficiently large $n$.
\end{proof}
Let $L=\{w\in D:d_C(w)>\frac{|C|}{2}\}$ and $J=D\setminus L$. Clearly $L\supseteq (I_1\cup I_2)\cap(I_1'\cup I_2')$ and $|L|\geq |(I_1\cup I_2)\cap(I_1'\cup I_2')|\geq 4(\frac{1}{6}-\e)n=(\frac{2}{3}-4\e)n$. Since $\delta(G[D])\geq \frac{1}{3}n-4f$, we have $|L|\leq |D|-\delta(G[D])\leq n-3f-(\frac{1}{3}n-4f)=\frac{2}{3}n+f$.
For each $w\in J$, $d_{J}(w)\leq 17\e n+|(I_1\cup I_2)\setminus (I_1'\cup I_2')|\leq 21\e n$ by Claims \ref{claim-vertex in I1I2 is critical for C} and \ref{claim-vertex in C is critical for C}.
For each edge $uv\in E(J)$, we have $d_J(u)+d_J(v)\leq42\e n$. Furthermore, $d_D(v)\geq\frac{1}{3}n-4f$ and $d_D(u)+d_D(v)\leq |L|+d_J(u)+d_J(v)\leq\frac{2}{3}n+f+42\e n$, so $d_D(u),d_D(v)\leq\frac{1}{3}n+5f+42\e n$. Let $S=\{w\in J:d_L(w)\leq(\frac{1}{3}+42\e)n+5f\}$. Both endpoints of every edge in $J$ lie in $S$. We next bound $|S|$ from above.

\begin{claim}\label{claim-upper bound of S}
It follows that $|S|\leq 3(\frac{2}{3}n+f-|L|)$.
\end{claim}
\begin{proof}
Since all vertices in $L$ are controlled by $C$, the set $L$ is independent and $e(L,J)\geq |L|(\frac{1}{3}n-4f)$. On the other hand, by the definition of $S$ we have
\[
e(L,J)\leq |S|(\frac{1}{3}n+5f+42\e n)+(n-3f-|L|-|S|)|L|=(n-3f-|L|)|L|-|S|(|L|-\frac{1}{3}n-42\e n-5f).
\]
Combining the upper and lower bounds gives
\[
|S|\leq (\frac{2}{3}n+f-|L|)\frac{|L|}{|L|-\frac{1}{3}n-42\e n-5f}\leq 3(\frac{2}{3}n+f-|L|).
\]
The last inequality follows from $|L|\geq \frac{2}{3}n-4\e n$ and $\e\leq \frac{1}{441}$.
\end{proof}
Recall that for each vertex $w\in Y$, we have 
\[
d_{L}(w)\geq d_{(I_1\cup I_2)\cap (I_1'\cup I_2')}(w)\geq d_{I_1\cup I_2}(w)-|(I_1\cup I_2)\setminus (I_1'\cup I_2')|\geq (\frac{1}{3}-6\e)n.
\]
Notice that for any $u\in N(w)\cap J$, $d_{L}(u)\leq |L|-d_{L}(w)\leq \frac{2}{3}n+f-(\frac{1}{3}-6\e )n=\frac{1}{3}n+6\e n+f$, which implies $N(w)\cap J\subseteq S$.

Let $Y_{\mathrm{hi}}$ denote the set of vertices $w\in Y$ satisfying $d_L(w)\geq \frac{1}{2}|L|-5f$. For every $w\in Y\setminus Y_{\mathrm{hi}}$, we have
\begin{equation}\label{ineq-lower bound on the degree in S}
d_{J}(w)=d_{S}(w)\geq \frac{1}{3}n-4f-d_{L}(w)>\frac{1}{3}n-4f-(\frac{1}{2}|L|-5f)\geq \frac{1}{2}(\frac{2}{3}n+f-|L|).
\end{equation}
We shall use the following separation observation in two iterative colorings.
\begin{claim}\label{claim-disjoint-S-neighborhoods}
Let $x,z\in Y$. If $z$ is not controlled by $N_L(x)$, then
\[
N_S(x)\cap N_S(z)=\emptyset.
\]
\end{claim}
\begin{proof}
Since $z$ is not controlled by $N_L(x)$,
\[
|N_L(z)\cap N_L(x)|\leq\frac{1}{2}d_L(x).
\]
Consequently,
\[
|N_L(z)\cup N_L(x)|
\geq d_L(z)+\frac{1}{2}d_L(x)
\geq\left(\frac{1}{2}-9\varepsilon\right)n.
\]
Suppose that $y\in N_S(z)\cap N_S(x)$. Triangle-freeness implies that $y$ has no neighbor in $N_L(z)\cup N_L(x)$. Since $|L|\leq\frac{2}{3}n+f$,
\[
d_L(y)\leq |L|-|N_L(z)\cup N_L(x)|
\leq\left(\frac{1}{6}+9\varepsilon\right)n+f.
\]
On the other hand, $y\in S\subseteq J$, so $d_J(y)\leq21\varepsilon n$ and
\[
d_L(y)=d_D(y)-d_J(y)
\geq\left(\frac{1}{3}-21\varepsilon\right)n-4f.
\]
For $\varepsilon=1/441$, these two inequalities are incompatible for all sufficiently large $n$, because $(\frac{1}{6}-30\varepsilon)n>5f$.
\end{proof}

We now show that $Y\setminus Y_{\mathrm{hi}}$ is $6$-colorable.
\begin{claim}\label{claim- Y-H has bounded chromatic number}
$Y\setminus Y_{\mathrm{hi}}$ is $6$-colorable.
\end{claim}
\begin{proof}
Set $Q=\frac{2}{3}n+f-|L|$ and start with $U_1=Y\setminus Y_{\mathrm{hi}}$. If $U_i\neq\emptyset$, choose $z_i\in U_i$, give one new color to all vertices of $U_i$ controlled by $N_L(z_i)$, and let $U_{i+1}$ be the set of vertices left uncolored. The colored set is independent and contains $z_i$. If $i<j$, then $z_j$ is not controlled by $N_L(z_i)$, so Claim~\ref{claim-disjoint-S-neighborhoods} gives
\[
N_S(z_i)\cap N_S(z_j)=\emptyset.
\]
By~\eqref{ineq-lower bound on the degree in S}, each selected representative satisfies $d_S(z_i)>Q/2$, whereas Claim~\ref{claim-upper bound of S} gives $|S|\leq3Q$. Six selected representatives would therefore have pairwise disjoint $S$-neighborhoods whose total order is greater than $3Q$, a contradiction. The process uses at most five colors; in particular, $Y\setminus Y_{\mathrm{hi}}$ is $6$-colorable.
\end{proof}
Now the size of $L$ can be bounded from below.
\begin{claim}\label{claim-lower bound on L}
We can assume $|L|\geq \frac{2}{3}n-9f$ at the cost of at most $61$ colors.
\end{claim}
\begin{proof}
First give one color to all vertices of $Y$ controlled by $L$, and let $U$ be the set of vertices left uncolored. Every $w\in U$ satisfies $d_L(w)\leq|L|/2$ and $N(w)\cap J\subseteq S$, and hence
 \[
d_{S}(w)\geq \frac{1}{3}n-4f-d_{L}(w)\geq \frac{1}{3}n-4f-\frac{1}{2}|L|=\frac{1}{2}(\frac{2}{3}n-8f-|L|).
\]
Suppose that $|L|\leq \frac{2}{3}n-9f$, and let $Q=\frac{2}{3}n+f-|L|$. Then $Q\geq10f$, and the preceding inequality together with Claim~\ref{claim-upper bound of S} gives
\[
d_S(w)\geq\frac{Q-9f}{2}\geq\frac{Q}{20}\geq\frac{|S|}{60}.
\]
Since $f$ is positive, $Q\geq10f>0$. Apply the iterative coloring procedure from the proof of Claim~\ref{claim- Y-H has bounded chromatic number} to $U$. Claim~\ref{claim-disjoint-S-neighborhoods} shows that the $S$-neighborhoods of the selected representatives are pairwise disjoint. If the procedure selects $m$ representatives, then
\[
\frac{mQ}{20}\leq \sum_{i=1}^{m}|N_S(z_i)|\leq |S|\leq3Q,
\]
and therefore $m\leq60$. Together with the first color used for vertices controlled by $L$, this gives a $61$-coloring of $Y$. Consequently, at the cost of at most $61$ colors, we may assume that $|L|\geq \frac{2}{3}n-9f$.
\end{proof}

We now use the hypothesis on $f$-vertex subgraphs of $KG(P,21f)$. Let
\[
p=\left\lceil\frac{1}{2}|L|-5f\right\rceil,
\qquad
\kappa=|L|-2p,
\qquad\text{and}\qquad
P=\left\lceil\frac{1}{3}n-10f\right\rceil.
\]
Then $10f-1\leq\kappa\leq10f$. The lower bound $|L|\geq\frac{2}{3}n-9f$ gives $p\geq P$, while
\[
2p+\kappa=|L|\leq\frac{2}{3}n+f\leq2P+21f.
\]
Embed the ground set of $KG(p,\kappa)$ into that of $KG(P,21f)$ and map every $p$-set to an arbitrary $P$-subset of it. Disjoint sets have disjoint images, so this defines a homomorphism from $KG(p,\kappa)$ to $KG(P,21f)$. Hence every subgraph of the former graph on at most $f$ vertices is $t$-colorable: a smaller subgraph may be enlarged to one on $f$ vertices and the resulting coloring restricted. For every $w\in Y_{\mathrm{hi}}$, we have $d_L(w)\geq p$; choose a $p$-subset of $N_L(w)$. Adjacent vertices have disjoint $L$-neighborhoods, so these choices define a homomorphism from $G[Y_{\mathrm{hi}}]$ to $KG(p,\kappa)$. Since $Y_{\mathrm{hi}}\subseteq Y\subseteq R$, $|R|=3f$, and $\kappa<p$ for all sufficiently large $n$, Proposition~\ref{prop-property of kneser graph}(\ref{property-kneser graph-factor not important}) with $C=3$ shows that $G[Y_{\mathrm{hi}}]$ is $(t+3)$-colorable. For clarity, the bound $3+9\cdot 2^{q_\varepsilon}$ in Claim~\ref{lem-neighbor in C is extremal} is cumulative and already includes the reductions made before that claim; the possible normalization in the paragraph after Claim~\ref{claim-exclude-fourth-neighborhood-pattern} adds at most $9+9\cdot 2^{q_\varepsilon}$ colors. Including the $12$ colors for $D$, the $61$ colors in Claim~\ref{claim-lower bound on L}, and the final $t+3$ colors gives
\[
12+(3+9\cdot 2^{q_\varepsilon})+(9+9\cdot2^{q_\varepsilon})+61+(t+3)=18\cdot 2^{q_\varepsilon}+88+t<10^{391}+t.
\]
This completes the proof.
\end{proof}

\subsection{The extremal structure with unbounded chromatic number}\label{subsection-structure of graph with large chromatic number and large minimum degree}
The proof of Theorem~\ref{thm-sufficiency} also yields Theorem~\ref{thm-24000}. A triangle-free graph with minimum degree at least $\frac{n}{3}-f$ and chromatic number at least $10^{391}$ admits a three-part decomposition resembling Hajnal's construction.
\begin{proof}[Proof of Theorem \ref{thm-24000}]
Assume that $\chi(G)\geq 10^{391}$. In the proof of Theorem~\ref{thm-sufficiency}, before Claim~\ref{claim-lower bound on L} is established, the required bounded coloring of $R$ has not yet been obtained. Hence $|L|\geq \frac{2}{3}n-9f$, and Claim~\ref{claim-upper bound of S} gives $|S|\leq30f$. Thus
\[
|J\setminus S|\geq n-3f-|L|-|S|\geq n-3f-|L|-3(\frac{2}{3}n+f-|L|)\geq \frac{1}{3}n-24f.
\]
Take a $(\frac{2}{3}n-9f)$-subset $I_1\subseteq L$ and a $(\frac{n}{3}-24f)$-subset $I_2\subseteq J\setminus S$. 
Let $A=V(G)\setminus(I_1\cup I_2)$; then $|A|=33f$. For every $w\in I_1$, the minimum-degree condition gives $d_{I_2}(w)\geq\delta(G)-|A|\geq\frac{1}{3}n-34f$. Consequently, every $w\in A$ satisfies either $d_{I_2}(w)\geq\frac{1}{3}n-34f$ or $d_{I_2}(w)\leq10f$. Indeed, otherwise some $w_0\in A$ would satisfy $10f<d_{I_2}(w_0)<\frac{1}{3}n-34f$, forcing $N(w_0)\cap I_1\ne\emptyset$. For $x\in N(w_0)\cap I_1$, triangle-freeness would then give $d_{I_2}(x)\leq|I_2|-|N(w_0)\cap I_2|<\frac{1}{3}n-34f$, a contradiction.
Partition $A$ into two subsets $A_1$ and $A_2$ by setting
$A_1=\{w\in A:d_{I_2}(w)\geq \frac{1}{3}n-34f\}$ and $A_2=\{w\in A:d_{I_2}(w)\leq 10f\}$. The set $A_1\cup I_1$ is controlled by $I_2$ and is therefore independent. For every $w\in A_2$, we have $d_{I_1\cup A_1}(w)\geq \frac{1}{3}n-44f$. Moreover, $|I_1\cup A_1|\leq n-\delta(G)\leq\frac{2}{3}n+f\leq2(\frac{1}{3}n-44f)+97f$. Choosing a $(\frac{1}{3}n-44f)$-subset of $N(w)\cap(I_1\cup A_1)$ for each $w\in A_2$ therefore gives a homomorphism from $G[A_2]$ to $KG(\frac{1}{3}n-44f,97f)$. Replacing $I_1$ by $I_1\cup A_1$ and $A$ by $A_2$ gives the stated partition and completes the proof.
\end{proof}
\section{Below the threshold $\frac{n}{3}$}\label{section-new upper and lower bound}
In this section, we derive Theorems~\ref{thm-slightly improve n/logn} and~\ref{thm-break the border} from Theorems~\ref{thm-lower bound on chi of kneser subgraph imply k3 free graph} and~\ref{thm-sufficiency}.
We first record two basic properties of Kneser graphs and their subgraphs.
\begin{proposition}\label{prop-property of kneser graph}
The following properties of the Kneser graph $KG(n,k)$ and its subgraphs hold:
\begin{enumerate}[(i)]
\item Let $g$ be the odd girth of Kneser graph $KG(n,k)$. Then $g\geq \frac{2n}{k}+1$.\label{property-kneser graph-odd girth large}
\item Suppose that $k<n$. Given a function $f(n)$, if every subgraph of $KG(n,k)$ on at most $f$ vertices is $s$-colorable for a fixed integer $s\geq1$, then, for every constant $C\geq1$, every subgraph of $KG(n,k)$ on at most $Cf$ vertices is $(s+\lceil\log_{3/2}C\rceil)$-colorable.\label{property-kneser graph-factor not important}
\end{enumerate}
\end{proposition}
\begin{proof}
For (\ref{property-kneser graph-odd girth large}), let $C_g$ be an odd cycle of $KG(n,k)$ of length $g$ and let $x$ be a vertex on this cycle. Let $u$ and $v$ be the two neighbors of $x$ on $C_g$. Walking along the cycle in the direction from $x$ to $v$, observe that for each $i$ the vertex at distance $2i$ from $x$ shares at least $n-ik$ elements with $x$. Taking $i=\frac{g-1}{2}$, the vertex at distance $g-1$ from $x$, namely $u$, shares at least $n-\frac{g-1}{2}k$ elements with $x$. Since $x$ and $u$ are adjacent, they must be disjoint, so $n-\frac{g-1}{2}k\leq0$, which gives $g\geq \frac{2n}{k}+1$.

For (\ref{property-kneser graph-factor not important}), the case $C=1$ is immediate. Let $q=\lceil\log_{3/2}C\rceil$ and let $A$ be a subgraph of $KG(n,k)$ on at most $Cf$ vertices. Since $k<n$, we have $\frac{n}{2n+k}>\frac13$. Averaging over the elements of $[2n+k]$, some independent star contains more than one third of the current vertices of $A$. Delete such a star and repeat, choosing the star anew at each step. After $q$ steps, at most $(2/3)^q|A|\leq f$ vertices remain, and the subgraph they induce is $s$-colorable by hypothesis. The deleted stars require $q$ additional colors, so $A$ is $(s+q)$-colorable.
\end{proof}
\subsection{A bound from Hajnal's construction}
In this subsection, we prove Theorem~\ref{thm-slightly improve n/logn}.
\begin{proof}[Proof of Theorem \ref{thm-slightly improve n/logn}]
Let $n\in \mathbb{N}$ and $k$ be a function of $n$ with $k|n$ and $k=o(n)$. Let $N=\frac{2n^2}{k(2n+k)}\binom{2n+k}{n}$. Consider the Kneser graph $KG(N,\frac{k}{n}N)$; clearly $KG(n,k)$ is a subgraph of $KG(N,\frac{k}{n}N)$. Note that $|V(KG(n,k))|=\binom{2n+k}{n}=\frac{k(2n+k)}{2n^2}N=\Theta(\frac{k}{n}N)$.
From the bounds $2^n \le \binom{2n+k}{n}\le \bigl(\frac{e(2n+k)}{n}\bigr)^n \le 6^n$ and the fact that $k=o(n)$, it follows that $n=\Theta(\log N)$. Without loss of generality that we assume $k$ is a function of $n$ that tends to infinity arbitrarily slowly as $n\to\infty$. Then $\frac{k}{n}N = w(N)\,\frac{N}{\log N}$, where $w(N)\to\infty$ as $N\to\infty$.
Observe that the chromatic number of $KG(n,k)$ is $k+2$, which tends to infinity.
By Theorem~\ref{thm-lower bound on chi of kneser subgraph imply k3 free graph} and Proposition~\ref{prop-property of kneser graph}(\ref{property-kneser graph-factor not important}) with suitable parameters, we can now construct the desired triangle-free graphs by regarding $KG(n,k)$ as a subgraph of $KG(N,\frac{k}{n}N)$.
\end{proof}
\subsection{Large minimum degree and bounded chromatic number}
In this subsection, we deduce Theorem~\ref{thm-break the border} from Theorem~\ref{thm-sufficiency}.
We use the following upper bound on the chromatic number of a graph of given order and odd girth.
\begin{theorem}[\cite{1984-szemeredi}]\label{thm-odd girth, chi, vertex number}
Let $t\geq1$ be an integer. If $G$ is an $s$-vertex graph with no odd cycle of length less than $4ts^{1/t}$, then $\chi(G)\leq t+1$.
\end{theorem}
\begin{proof}[Proof of Theorem \ref{thm-break the border}]
Denote
\[
q=\left\lceil\frac{(1+\varepsilon)(1-\delta)}{\delta}\right\rceil.
\]
Then $q$ is a positive integer and $q>(1-\delta)/\delta$, so
\[
\gamma:=\delta-\frac{1-\delta}{q}>0.
\]
Set
\[
a_n=\left\lceil n^{1-\delta}\right\rceil,
\qquad
P_n=\left\lceil\frac{1}{3}n-10a_n\right\rceil.
\]
The function $a_n$ is positive, integer-valued, and $o(n)$. By Proposition~\ref{prop-property of kneser graph}(\ref{property-kneser graph-odd girth large}), the odd girth $g_n$ of $KG(P_n,21a_n)$ satisfies
\[
g_n\geq\frac{2P_n}{21a_n}+1=\Omega(n^\delta).
\]
Moreover, $g_n/a_n^{1/q}$ tends to infinity at least as a positive constant times $n^\gamma$. Hence $g_n\geq4q a_n^{1/q}$ for all sufficiently large $n$. Let $F$ be any subgraph of $KG(P_n,21a_n)$ on at most $a_n$ vertices and denote $h=|V(F)|$. The odd girth of $F$ is at least $g_n\geq4q h^{1/q}$, so Theorem~\ref{thm-odd girth, chi, vertex number}, applied with $(s,t)=(h,q)$, shows that $F$ is $(q+1)$-colorable. Finally, $\delta(G)\geq n/3-n^{1-\delta}\geq n/3-a_n$, and Theorem~\ref{thm-sufficiency}, applied with $f(n)=a_n$, gives $\chi(G)\leq 10^{391}+(q+1)=10^{391}+1+q$, as required.
\end{proof}


\section{Extension to $K_r$-free graphs}\label{section-connect Kr-free and Kr-1-free graphs}
In this section, we extend the preceding results to $K_r$-free graphs by strengthening results of Brandt \cite{2001-brandt} and Goddard and Lyle \cite{2011-lyle}.
We use the following theorem of Goddard and Lyle \cite{2011-lyle}.
\begin{theorem}[\cite{2011-lyle}]\label{thm-lyle's thm}
Let $r\geq4$ and $k\in\{2,\dots,r-2\}$, and set $\delta_{r,k}=(k(r-3)+1)/(k(r-2)+1)$. Let $G$ be a maximal $K_r$-free graph with $\delta(G)>\delta_{r,k}n$, and let $I$ be a maximum independent set in $G$. Then $G[V(G)\setminus N(u)]$ is $K_k$-free for every $u\in I$.
\end{theorem}
\begin{proof}[Proof of Theorem \ref{thm-decomposition of Kr-free graph with large minimum degree}]
Let $I$ be a maximum independent set of $G$, fix $u\in I$, and write $f=f(n)$ and $\delta=\delta(G)$. Since
\[
\frac{2r-5}{2r-3}-\frac{3r-8}{3r-5}
=\frac{1}{(2r-3)(3r-5)}>0,
\]
Theorem~\ref{thm-lyle's thm} with $k=3$ and $r\geq5$ shows, for all sufficiently large $n$, that $G[V(G)\setminus N(u)]$ is triangle-free. 
If $V(G)\setminus N(u)$ is independent, the $K_r$-free maximality argument as shows that it is complete to $N(u)$; since $G[N(u)]$ is $K_{r-1}$-free, the first outcome again holds.

Suppose that $G[V(G)\setminus N(u)]$ contains an edge $s_1s_2$. Extend this edge to a copy $S=\{s_1,\dots,s_{r-2}\}$ of $K_{r-2}$. This is possible because $\delta(G)\geq \frac{2r-5}{2r-3}n-f>\frac{r-3}{r-2}n$, so the vertices can be chosen greedily.
Set $I'=\bigcap_{s_i\in S}N(s_i)\cap N(u)$. When $r\geq5$, $G[V(G)\setminus N(u)]$ is $K_3$-free and we have $\bigcap_{s_i\in S}N(s_i)\subseteq N(u)$. Moreover, $I'$ is independent because it is the common neighborhood of a copy of $K_{r-2}$ in the $K_r$-free graph $G$. The maximality of $I$ gives $|I'|\leq|I|$. Since $|N(s_i)\cap N(u)|+|N(s_i)\setminus N(u)|=d(s_i)$ and $|N(u)|=d(u)$, we obtain
\begin{equation}\label{ineq-lower bound of I'}
|I'|\geq (r-2)\delta(G)-\sum_{s_i\in S}|N(s_i)\setminus N(u)|-(r-3)d(u).
\end{equation}
Since $I$ is maximal, every vertex outside $I$ has a neighbor in $I$. We may therefore choose the edge $s_1s_2$ with $s_1\in I$ and $s_2\in V(G)\setminus(N(u)\cup I)$.
Since $u$ is not adjacent to $s_2$ and $G$ is maximal $K_r$-free, $N(u)\cap N(s_2)$ contains a copy $T$ of $K_{r-2}$. Thus $u$ and $s_2$ are common neighbors of $T$. The common neighborhood of $T$ has order at least $(r-2)\delta-(r-3)n$ and is independent; moreover, it is contained in $V(G)\setminus N(u)$ because $N(u)$ is $K_{r-1}$-free. Hence $|N(s_2)\setminus N(u)|\leq(n-d(u))-((r-2)\delta-(r-3)n)$. Since $s_1\in I$, we also have $|N(s_1)\setminus N(u)|\leq(n-d(u))-|I|$. Finally, $|N(s_i)\setminus N(u)|\leq n-d(u)$ for $3\leq i\leq r-2$.
Substituting these estimates into~\eqref{ineq-lower bound of I'} gives
\begin{align*}
|I'|&\geq (r-2)\delta(G)-\sum_{s_i\in S}|N(s_i)\setminus N(u)|-(r-3)d(u)\\
&\geq (r-2)\delta -[(r-2)(n-d(u))-|I|-((r-2)\delta-(r-3)n)]-(r-3)d(u)\\
&\geq (r-2)\delta+\delta -(r-2)n+(r-2)\delta-(r-3)n+|I|\\
&=(2r-3)\delta -(2r-5)n+|I|.
\end{align*}
Since $\delta\geq \frac{2r-5}{2r-3}n-f$, we have $|I|-(2r-3)f\leq |I'|\leq |I|$, and all the corresponding terms that appear in the inequalities differ from the estimated bounds by at most $(2r-3)f$; we call such inequalities \textit{critical inequalities}. For example, $\delta\leq d(u)\leq \delta+(2r-3)f$ is a critical inequality for $d(u)$. Here we list all the critical inequalities in Table \ref{tab:critical ineq} for verifying easily. 

\begin{table}[]
    \centering
    \begin{tabular}{ccc}
  parameters & lower bound & upper bound \\
  $|I'|$ & $|I|-(2r-3)f$ & $|I|$\\
  $d(u),d(s_i)$ & $\delta$ & $\delta+(2r-3)f$ \\
  $|N(s_2)\setminus N(u)|$ & \small{$(n-d(u))-((r-2)\delta-(r-3))-(2r-3)f$}  & \small{$(n-d(u))-((r-2)\delta-(r-3))$} \\
  $|N(s_1)\setminus N(u)|$ & $(n-d(u))-|I|-(2r-3)f$  & $(n-d(u))-|I|$\\
  $|N(s_i)\setminus N(u)|$ for $i\geq 3$ & $n-d(u)-(2r-3)f$  & $n-d(u)$ 
\end{tabular}
    \caption{The critical inequalities}
    \label{tab:critical ineq}
\end{table}

Since $T$ induces a $K_{r-2}$, the critical inequality for $|N(T)|$ shows that $T$ has between $(r-2)\delta-(r-3)n$ and $(r-2)\delta-(r-3)n+(2r-3)f$ common neighbors. Thus all but at most $(2r-3)f$ vertices in $N(u)$ are adjacent to exactly $r-3$ vertices of $T$. Consequently, $N(u)$ admits a partition $N(u)=A_1\cup\dots\cup A_{r-2}\cup R_1$ such that each $A_i$ is independent and $|R_1|\leq(2r-3)f$. Relabel the sets so that $|A_1|\geq\dots\geq|A_{r-2}|$.
\begin{claim}
For each $i\in [r-2]$, $|A_i|\geq \frac{n}{2r-3}-(3r-5)f$.
\end{claim}
\begin{proof}
As $\delta\geq \frac{2r-5}{2r-3}n-f$, we have $|A_i|\leq \frac{2}{2r-3}n+f$. This implies that
\begin{align*}
|A_i|&=d(u)-|R_1|-|A_1|-\dots-|A_{i-1}|-|A_{i+1}|-\dots-|A_{r-2}|\\
&\geq \frac{2r-5}{2r-3}n-f -(2r-3)f-(r-3)(\frac{2}{2r-3}n+f)\\
&=\frac{n}{2r-3}-(3r-5)f.
\end{align*}
\end{proof}
\begin{claim}\label{claim-lower bound on I}
$|I|\geq (2-\frac{1}{r-2})\frac{n}{2r-3}-\frac{2(r-1)}{r-2}f$.
\end{claim}
\begin{proof}
By averaging,
\[
|A_1|\geq \frac{1}{r-2}(d(u)-|R_1|)\geq \frac{1}{r-2}(\frac{2r-5}{2r-3}n-f-(2r-3)f)=(2-\frac{1}{r-2})\frac{n}{2r-3}-\frac{2(r-1)}{r-2}f.
\]
The maximality of $I$ now gives the desired bound.
\end{proof}
Let $R_2=V\setminus (N(u)\cup I)$.
Note that the properties of $G$ established so far can also be obtained by considering a different choice of the tuple of vertices and set $(u,s_1,\dots,s_{r-2},T)$. We call these properties \textit{extremal properties}.
\begin{claim}
$|R_2|\leq (4r+20)f$.
\end{claim}
\begin{proof}
Since $|N(s_2)\setminus N(u)|\leq (n-d(u))-((r-2)\delta-(r-3)n)$, we have
\[
|N(s_2)\cap N(u)|\geq \delta-(n-d(u))+((r-2)\delta-(r-3)n)\geq r\delta-(r-2)n=\frac{2(r-3)}{2r-3}n-rf.
\]
Any copy of $K_{r-2}$ in $N(s_2)\cap N(u)$, together with $u$ and $S$, satisfies the extremal properties. Recall that, for $i\in[r-2]$, the set $A_i$ consists of the common neighbors of $r-3$ vertices of $T$. Let $T=\{t_1,\dots,t_{r-2}\}$ with $t_i\in A_i$. For every $w_i\in A_i\cap N(s_2)$, the set $(T\setminus\{t_i\})\cup\{w_i\}$ also induces a $K_{r-2}$ in $N(u)\cap N(s_2)$, so the preceding analysis applies again.
This implies that any vertex in $N(s_2)\cap A_i$ is adjacent to all but at most $(2r-3)f$ vertices in $N(u)\setminus A_i$.

Then we consider the common neighbors of $s_1$ and $s_2$ in $N(u)$. Clearly $N(s_1)\cap N(s_2)\cap N(u)$ is $K_{r-2}$-free, and
\begin{align}
\nonumber|N(s_1)\cap N(s_2)\cap N(u)|&\geq d(s_1)+d(s_2)-d(u)-|N(s_1)\setminus N(u)|-|N(s_2)\setminus N(u)|\\
&\nonumber \geq \delta-(2r-3)f-(n-d(u))+|I|-(n-d(u))+((r-2)\delta-(r-3)n)\\
&\nonumber\geq (r+1)\delta-(r-1)n-(2r-3)f+|I|\\
&=\frac{2(r-4)}{2r-3}n-(3r-2)f+|I|. \label{ineq-lower bound on the common neighbors in N(u)}
\end{align}
If $|N(s_1)\cap N(s_2)\cap N(u)|>|A_1|+\dots+|A_{r-3}|+|R_1|+(r-3)(2r-3)f$, then we can find a $K_{r-2}$ in $N(s_1)\cap N(s_2)\cap N(u)$, because every vertex in $N(s_2)\cap(A_1\cup\dots\cup A_{r-2})$ is adjacent to all but at most $(2r-3)f$ vertices in $N(u)$. Hence $|N(s_1)\cap N(s_2)\cap N(u)|\leq |A_1|+\dots+|A_{r-3}|+|R_1|+(r-3)(2r-3)f\leq |A_1|+\dots+|A_{r-3}|+(r-2)(2r-3)f$.
At the same time, 
\[
|A_1|+\dots+|A_{r-3}|\leq (r-3)|I|.
\]
Combining these inequalities with the upper and lower bounds on $|N(s_1)\cap N(s_2)\cap N(u)|$ gives
\[
(r-4)|I|\geq \frac{2(r-4)}{2r-3}n-2(r^2-2r+2)f.
\]
When $r\geq 5$, this implies $|I|\geq \frac{2}{2r-3}n-\frac{2(r^2-2r+2)}{r-4}f$ and
\[
|R_2|= (n-d(u))-|I|\leq n-\frac{2r-5}{2r-3}n+f-(\frac{2}{2r-3}n-\frac{2(r^2-2r+2)}{r-4}f)=\frac{r(2r-3)}{r-4}f\leq (4r+15)f.
\]
When $r=4$, without loss of generality we can assume $t_1\in I'$ and $t_2\in (N(s_2)\cap N(u))\setminus I'$. We show this case by proving that $E(I',I'^c)$ is dense and thus $|I'|$ must be roughly $\frac{2}{2r-3}n$ and $|R_2|$ is small.

We first show $E(I',N(u)\setminus I')$ is dense. As above we have $|N(t_2)\cap I'|\geq |I'|-(2r-3)f$. Then we can consider a new $T'=\{t_1',t_2\}$ with $t_1'\in N(t_2)\cap I'$. All the extremal properties hold for new $T'$ and we also have $|N(t_1')\cap N(u)|\geq d(u)-|I'|-(2r-3)f$. Notice that there are at least $|I'|-(2r-3)f$ such $t_1'$ in $I'$.

We next show that $E(I',R_2)$ is dense. The critical inequality for $|N(s_1)\setminus N(u)|$ gives $|N(s_1)\setminus N(u)|\geq(n-d(u))-|I|-(2r-3)f$. For $s_2'\in N(s_1)\setminus N(u)$, consider the new pair $(s_1,s_2')$ and the corresponding $T'$. The extremal properties give either $|(N(s_1)\cap N(s_2'))\cap I'|\geq |I'|-5(r-1)f$ or $|(N(s_1)\cap N(s_2'))\cap I'|\leq(2r-3)f$. Suppose, to the contrary, that
\[
(2r-3)f<|(N(s_1)\cap N(s_2'))\cap I'|< |I'|-5(r-1)f.
\]
Then $N(s_1)\cap N(s_2')$ is independent, and it follows from (\ref{ineq-lower bound on the common neighbors in N(u)}) that $|(N(s_2')\cap N(s_1))\cap (N(u)\setminus I')|> (2r-3)f$ and $|(N(s_1)\cap N(s_2'))\cap I'|>(2r-3)f$, contradicting the fact that all but at most $(2r-3)f$ vertices in $I'$ have at least $d(u)-|I'|-(2r-3)f$ neighbors in $N(u)\setminus I'$. Note that $N(s_1)\cap I'=I'$. If $|(N(s_1)\cap N(s_2'))\cap I'|=|N(s_2')\cap I'|\leq (2r-3)f$, then by the critical inequality for $|N(s_1)\cap N(s_2')|$ we have $|N(s_1)\cap N(s_2')|\geq |I|-(2r-3)f$ and $|(N(s_1)\cap N(s_2'))\setminus I'|\geq |I|-5(r-1)f$, which leads to $|N(s_1)\cap N(u)|\geq |I'|+|I|-5(r-1)f\geq 2|I|-(7r-8)f$. By Claim \ref{claim-lower bound on I}, $|I|\geq (2-\frac{1}{r-2})\frac{n}{2r-3}-\frac{2(r-1)}{r-2}f=\frac{3}{10}n-3f$. Hence we have
\[
d(s_1)-|N(s_1)\setminus N(u)|=|N(s_1)\cap N(u)|\geq \frac{3}{5}n-26f.
\]
By the critical inequality for $d(s_1)$, $d(s_1)\leq \delta+(2r-3)f=\frac{2r-5}{2r-3}n+2(r-2)f=\frac{3}{5}n+4f$. The preceding inequality therefore gives $|N(s_1)\setminus N(u)|\leq30f$. On the other hand, the critical inequality for $|N(s_1)\setminus N(u)|$ gives $|N(s_1)\setminus N(u)|\geq n-d(u)-|I|-(2r-3)f$. Hence $|I|\geq n-d(u)-35f$ and $|R_2|\leq35f$, proving the claim. We may therefore assume that $|(N(s_1)\cap N(s_2'))\cap I'|\geq |I'|-5(r-1)f$ for every $s_2'\in N(s_1)\setminus N(u)$.

Finally, we show that $E(I',I)$ is dense. For $s_1'\in I$, if $N(s_1')\subseteq N(u)$, then $|N(s_1')\cap N(u)|\geq\delta\geq d(u)-(2r-3)f$ and $|N(s_1')\cap I'|\geq|I'|-(2r-3)f$.
If $N(s_1')\setminus N(u)\ne\emptyset$, choose $s_2'\in N(s_1')\setminus N(u)$ and obtain the extremal properties with the corresponding $T'$. The remainder of the argument is as in the preceding paragraph.
If $|R_2|\leq2(2r-3)f$, the claim follows. Otherwise, the critical inequalities for $|N(s_1)\setminus N(u)|$ and $|N(s_1')\setminus N(u)|$ allow us to choose $s_2'\in N(s_1)\cap N(s_1')$ with $|N(s_2')\cap I'|\geq|I'|-5(r-1)f$. Again, either $|(N(s_1')\cap N(s_2'))\cap I'|\geq|I'|-5(r-1)f$ or $|(N(s_1')\cap N(s_2'))\cap I'|\leq(2r-3)f$. The second alternative is impossible: by~\eqref{ineq-lower bound on the common neighbors in N(u)} and Claim~\ref{claim-lower bound on I}, $|N(s_1')\cap N(s_2')|\geq |I|-9f\geq\frac{3}{10}n-12f$, while the critical inequality for $|I'|$ and Claim~\ref{claim-lower bound on I} give $|N(s_2')\cap I'|\geq|I'|-5(r-1)f\geq|I|-(7r-8)f\geq\frac{3}{10}n-23f$, and
\begin{align*}
|N(s_2')\cap N(u)|&\leq d(s_2')-|N(s_2')\setminus N(u)|\\
&\leq \delta+(2r-3)f-(n-d(u))+(r-2)\delta-(r-3)n+(2r-3)f\\
&\leq  \frac{2}{5}n+11f.
\end{align*}
Hence for any such $s_1'$ we have $|N(s_1')\cap I'|\geq|(N(s_1')\cap N(s_2'))\cap I'|\geq |I'|-5(r-1)f$.

Altogether, $E(I',I'^c)$ is dense. More precisely,
\begin{align*}
&\quad e(I',N(u)\setminus I')+e(I',R_2)+e(I',I)\\
&\geq (|I'|-5f)(d(u)-|I'|-5f)+(|R_2|-5f)(|I'|-15f)+|I|(|I'|-15f)\\
&=|I'|(n-|I'|)-5f(3n-2d(u)+|I'|)+100f^2.
\end{align*}
On the other hand, by the critical inequality for $d(t_1)$, and all but at most $5f$ vertices in $I'$ can serve as such a $t_1$ in $T$, we have
\begin{align*}
e(I',N(u)\setminus I')+e(I',I)+e(I',R_2)&=\sum_{w\in I'}d(w)\\
&\leq (|I'|-5f)(\delta+5f)+5f\cdot (n-|I'|)\\
&= |I'|(\frac{3}{5}n-f)+5f(\frac{2}{5}n+f)-25f^2.
\end{align*}
Combining the upper and lower bounds on $e(I',N(u)\setminus I')+e(I',I)+e(I',R_2)$ gives
\[
|I'|(\frac{2}{5}n-|I'|+f)\leq 5f(\frac{17}{5}n+|I'|-2d(u)+f)-125f^2\leq 5f(\frac{11}{5}n+|I'|+3f)-125f^2\leq 13fn,
\]
which implies $|I'|\geq\frac{2}{5}n-35f$ and hence $|I|\geq\frac{2}{5}n-35f$. Therefore $|R_2|\leq n-d(u)-|I|\leq36f\leq(4r+20)f$, completing the proof.
\end{proof}
We next estimate $|A_i|$ for $i\in[r-2]$. The minimum-degree condition gives $|A_i|\leq n-\delta=\frac{2}{2r-3}n+f$. Recall that $N(u)=A_1\cup\dots\cup A_{r-2}\cup R_1$ and $|A_1|\geq\dots\geq|A_{r-2}|$.
\begin{claim}
For $i\in [r-3]$, $\frac{2}{2r-3}n-(2r^2+3r-10)f\leq |A_i|\leq \frac{2}{2r-3}n+f$; $\frac{n}{2r-3}-(3r-5)f\leq |A_{r-2}|\leq \frac{n}{2r-3}+2(r^2-2) f$.
\end{claim}
\begin{proof}
For $s_1 s_2\in E(G\setminus N(u))$ with $s_1\in I$ and $s_2\in V(G)\setminus (I\cup N(u))$, $N(s_1)\cap N(s_2)$ is $K_{r-2}$-free. Recall that $T=\{t_1,\dots,t_{r-2}\}$, $A_i=\{v\in N(u):N(v)\cap T=T\setminus\{t_i\}\}$, $t_i\in A_i$, and every vertex in $N(s_2)\cap A_i$ is adjacent to all the vertices in $N(u)\setminus A_i$ except at most $(2r-3)f$ vertices. Recall that the upper bound on $|N(s_2)\setminus N(u)|$ is
\[
|N(s_2)\setminus N(u)|\leq (n-d(u))-(r-2)\delta+(r-3)n\leq \frac{n}{2r-3}+(r-1)f,
\]
and
\[
|N(s_2)\cap N(u)|\geq \delta-|N(s_2)\setminus N(u)|\geq \frac{2(r-3)}{2r-3}n-rf.
\]
Since $|N(s_1)\cap N(u)|\geq \delta-|R_2|\geq d(u)-(2r-3)f-(4r+20)f=d(u)-(6r+17)f$, we derive $|N(s_1)\cap N(s_2)\cap N(u)|\geq \frac{2(r-3)}{2r-3}n-(7r+17)f$.
If $|A_{r-2}|> \frac{n}{2r-3}+2(r^2-2) f$, then we have for each $i\in [r-2]$, $|A_i|> \frac{n}{2r-3}+2(r^2-2)f$, and
\begin{align*}
|N(s_1)\cap N(s_2)\cap A_i|&\geq |N(s_1)\cap N(s_2)\cap N(u)|-(d(u)-|A_i|)\\
&\geq |N(s_1)\cap N(s_2)\cap N(u)|-(\delta+(2r-3)f)+|A_i|\\
&>(r-3)(2r-3)f.
\end{align*}
We could then greedily find a $K_{r-2}$ in $N(s_1)\cap N(s_2)\cap(A_1\cup\dots\cup A_{r-2})$, a contradiction. Hence $|A_{r-2}|\leq \frac{n}{2r-3}+2(r^2-2)f$, which implies
\begin{align*}
|A_{1}|\geq\dots\geq |A_{r-3}|&\geq d(u)-|A_1|-\dots-|A_{r-4}|-|A_{r-2}|-|R_{1}|\\
&\geq (\frac{2r-5}{2r-3}n-f)-(r-4)(\frac{2}{2r-3}n+f)-(\frac{n}{2r-3}+2(r^2-2) f)-(2r-3)f\\
&=\frac{2}{2r-3}n-(2r^2+3r-10)f.
\end{align*}
\end{proof}
It remains to assign the vertices in $R_1\cup R_2$ to the desired parts.
By the minimum-degree condition, for each $i\in[r-3]$, every vertex in $A_i$ is adjacent to all but at most $O(f)$ vertices outside $A_i$. Consequently, every $w\in R_1\cup R_2$ satisfies either $|N(w)\cap I|=O(f)$ or $|N(w)\cap A_{i_w}|=O(f)$ for some $i_w$; otherwise, $N(w)$ would contain a copy of $K_{r-1}$.
For each $w\in R_1\cup R_2$, choose one qualifying index from the preceding alternatives, breaking ties arbitrarily: use index $0$ when $|N(w)\cap I|=O(f)$ and index $i\in[r-2]$ when $|N(w)\cap A_i|=O(f)$. Let $M_i$ be the set of vertices assigned index $i$. Then $M_0,M_1,\dots,M_{r-2}$ partition $R_1\cup R_2$. One readily verifies that $A_i\cup M_i$ is independent for every $i\in[r-3]$, that $E(M_{r-2},A_{r-2})=\emptyset$, and that $I\cup M_0$ is independent.
Let $I_i=A_i\cup M_i$ for $i\in [r-3]$, let $I_{r-2}=I\cup M_0$, and set $I_0=A_{r-2}$ and $A=M_{r-2}$.
It remains to show that $M_{r-2}$ is homomorphic to the Kneser graph $KG(\frac{n}{2r-3}-O(f),O(f))$.
For an edge $vw\in E(M_{r-2})$, $|N(v)\cap A_{r-2}|=O(f)$, $|N(w)\cap A_{r-2}|=O(f)$. 
Since $N(v)\cap N(w)$ is $K_{r-2}$-free, we have $|N(v)\cap N(w)|\leq \frac{2(r-3)}{2r-3}n+O(f)$.
As $d(v),d(w)\geq \delta=\frac{2r-5}{2r-3}n-f$,
\[
|N(v)\cap N(w)|\geq 2\delta-(n-|A_{r-2}|)-O(f)\geq \frac{2(r-3)}{2r-3}n-O(f).
\]
Hence $|N(v)\cap N(w)|=\frac{2(r-3)}{2r-3}n+O(f)$, and there exists $i_{v,w}$ such that
\[
|N(v)\cap N(w)\cap (\cup_{j\in[r-2],j\neq i_{v,w}}I_{j})|\geq \frac{2(r-3)}{2r-3}n-O(f).
\]
Furthermore, $|N(v)\cap I_{i_{v,w}}|\geq \frac{n}{2r-3}-O(f)$, $|N(w)\cap I_{i_{v,w}}|\geq \frac{n}{2r-3}-O(f)$, and $N(v)\cap N(w)\cap I_{i_{v,w}}=\emptyset$. Indeed, otherwise we could greedily construct a $K_r$, starting with the triangle formed by $v,w$ and a common neighbor in $I_{i_{v,w}}$. We may consider each component of $M_{r-2}$ separately; within a component, $i_{v,w}$ is constant along its edges, so assume that $i_{v,w}=1$. For each $w\in M_{r-2}$, $|N(w)\cap I_1|\geq \frac{n}{2r-3}-O(f)$, while $N(v)\cap N(w)\cap I_1=\emptyset$ for every edge $vw\in E(M_{r-2})$. Thus $M_{r-2}$ is homomorphic to $KG(\frac{n}{2r-3}-O(f),O(f))$, completing the proof.
\end{proof}
We can now extend the results of the preceding section to $K_r$-free graphs for $r\geq3$.

\begin{theorem}\label{thm-Kr-free graph with large minimum degree has bounded chi}
Let $r\geq3$ be an integer, let $0<\delta<1$, and let $\varepsilon>0$. Denote
\[
q=\left\lceil\frac{(1+\varepsilon)(1-\delta)}{\delta}\right\rceil.
\]
For all sufficiently large $n$, every $n$-vertex maximal $K_r$-free graph with minimum degree at least $\frac{2r-5}{2r-3}n-n^{1-\delta}$ has chromatic number at most $10^{391}+(r-2)+q$.
\end{theorem}
\begin{proof}
We proceed by induction on $r$. The case $r=3$ is Theorem~\ref{thm-break the border}. Denote $a_n=\lceil n^{1-\delta}\rceil$. Since the stated degree condition implies $\delta(G)\geq\frac{2r-5}{2r-3}n-a_n$, Theorem~\ref{thm-decomposition of Kr-free graph with large minimum degree}, applied with the positive integer-valued function $a_n$, shows that either\\ CASE 1. $G$ is the join of an independent set $I$ and a $K_{r-1}$-free graph $G_{r-1}$, or\\
CASE 2. $V(G)$ admits a partition $V(G)=I_0\cup I_1\cup\dots\cup I_{r-2}\cup A$ such that $\bigl||I_0|-\frac{n}{2r-3}\bigr|=O(a_n)$; for every $i\in[r-2]$, $\bigl||I_i|-\frac{2n}{2r-3}\bigr|=O(a_n)$; $|A|=O(a_n)$; and $G[A]$ admits a homomorphism to $KG(\frac{n}{2r-3}-O(a_n),O(a_n))$.

For the first case (CASE 1), we have
\begin{align*}
\delta(G_{r-1})&\geq \delta(G)-|I|\\
&=\frac{2r-7}{2r-5}(n-|I|)-\frac{2r-7}{2r-5}n-\frac{2}{2r-5}|I|+\delta(G)\\
&\geq \frac{2r-7}{2r-5}(n-|I|)-\frac{2r-7}{2r-5}n-\frac{2}{2r-5}(n-\delta(G))+\delta(G)\\
&\geq \frac{2r-7}{2r-5}|G_{r-1}|-n+\frac{2r-3}{2r-5}(\frac{2r-5}{2r-3}n-a_n)\\
&=\frac{2r-7}{2r-5}|G_{r-1}|-\frac{2r-3}{2r-5}a_n\\
&\geq \frac{2r-7}{2r-5}|G_{r-1}|-|G_{r-1}|^{1-\delta'},
\end{align*}
where we choose $0<\delta'<\delta$ sufficiently close to $\delta$ such that $q>(1-\delta')/\delta'$, and then choose $\varepsilon'>0$ so small that
\[
\left\lceil\frac{(1+\varepsilon')(1-\delta')}{\delta'}\right\rceil\leq q.
\]
Since $I\neq\emptyset$, every vertex $x\in I$ has $d_G(x)=|V(G_{r-1})|$; hence $|V(G_{r-1})|\geq\delta(G)=\Theta(n)$. The last inequality in the display above therefore holds for all sufficiently large $n$ because $\delta'<\delta$. After adding edges within $G_{r-1}$ if necessary, the induction hypothesis with $(\delta',\varepsilon')$ gives
\[
\chi(G_{r-1})\leq 10^{391}+(r-3)+q.
\]
One additional color for $I$ proves the asserted bound in this case.

For the second case (CASE 2), there are constants $C,c>0$, depending only on $r$, such that $|A|\leq Ca_n\leq 2Cn^{1-\delta}$ and Proposition~\ref{prop-property of kneser graph}(\ref{property-kneser graph-odd girth large}) gives odd girth at least $cn^\delta$ for the target Kneser graph. A homomorphism $G[A]\to K$ maps every odd cycle of $G[A]$ to an odd closed walk in $K$, which contains an odd cycle no longer than that walk. Thus the odd girth of $G[A]$ is at least that of $K$. Since $q>(1-\delta)/\delta$, for all sufficiently large $n$ we have
\[
cn^\delta\geq4q(2Cn^{1-\delta})^{1/q}\geq4q|A|^{1/q}.
\]
Theorem~\ref{thm-odd girth, chi, vertex number}, applied to $G[A]$ with the integer parameter $q$, now shows that $G[A]$ is $(q+1)$-colorable. The $r-1$ independent parts require at most $r-1$ further colors, so $\chi(G)\leq r+q<10^{391}+(r-2)+q$.
\end{proof}

\begin{theorem}\label{thm-Kr-free case-extremal graph}
Fix an integer $r\geq3$, and let $f$ be a positive integer-valued function satisfying $f(n)=o(n)$. For all sufficiently large $n$, if $G$ is an $n$-vertex $K_r$-free graph with minimum degree at least $\frac{2r-5}{2r-3}n-f(n)$ and chromatic number at least $10^{391}+(r-3)$, then $G$ admits a partition $V(G)=I_0\cup I_1\cup\dots\cup I_{r-2}\cup A$ such that $I_0,I_1,\dots,I_{r-2}$ are independent, $\bigl||I_0|-\frac{n}{2r-3}\bigr|=O(f(n))$, $\bigl||I_i|-\frac{2n}{2r-3}\bigr|=O(f(n))$ for every $i\in[r-2]$, and $G[A]$ admits a homomorphism to $KG(\frac{n}{2r-3}-O(f(n)),O(f(n)))$. All implicit constants may depend on $r$ only.
\end{theorem}
\begin{proof}
We proceed by induction on $r$. The case $r=3$ is Theorem~\ref{thm-24000}.
By adding edges if necessary, assume that $G$ is maximal $K_r$-free.
Write $f=f(n)$. By Theorem~\ref{thm-decomposition of Kr-free graph with large minimum degree}, either $G$ is the join of an independent set $I$ and a $K_{r-1}$-free graph $G_{r-1}$, or $V(G)$ admits a partition $V(G)=I_0\cup I_1\cup \dots\cup I_{r-2}\cup A$ such that $\bigl||I_0|-\frac{n}{2r-3}\bigr|=O(f)$; for each $i\in [r-2]$, $\bigl||I_i|-\frac{2n}{2r-3}\bigr|=O(f)$; $|A|=O(f)$, and $G[A]$ admits a homomorphism to a Kneser graph $KG(\frac{n}{2r-3}-O(f),O(f))$.

The second case yields the result directly. It remains to justify that the first case also yields the conclusion with an error measured at the original order $n$. Suppose, to the contrary, that there is a sequence of maximal counterexamples $(G_j)_{j\geq1}$ in the first case, where $n_j=|V(G_j)|\to\infty$. Let
\[
G_j=I_j\vee H_j,\qquad m_j=|V(H_j)|,\qquad f_j=f(n_j),
\]
where $I_j$ is a nonempty independent set and $H_j$ is $K_{r-1}$-free. Since $\chi(G_j)\geq 10^{391}+(r-3)$, we have $\chi(H_j)\geq10^{391}+(r-4)$, and the calculation used above gives
\[
\delta(H_j)\geq\frac{2r-7}{2r-5}m_j-c_rf_j,
\qquad c_r:=\frac{2r-3}{2r-5}.
\]
Every vertex of $I_j$ has degree $m_j$, so $m_j\geq\delta(G_j)=\Theta(n_j)$. Passing to a subsequence, we may assume that $(m_j)$ is strictly increasing. Define $g(m_j)=\lceil c_rf_j\rceil$ and $g(t)=1$ at all other positive integers. Then $g$ is positive and integer-valued, $g(t)=o(t)$, and $g(m_j)=O(f_j)$. The induction hypothesis applied to $H_j$ with this error function gives, for all sufficiently large $j$, a partition
\[
V(H_j)=J_{0,j}\cup J_{1,j}\cup\dots\cup J_{r-3,j}\cup A_j
\]
such that
\[
\left||J_{0,j}|-\frac{m_j}{2r-5}\right|=O(f_j),
\qquad
\left||J_{i,j}|-\frac{2m_j}{2r-5}\right|=O(f_j)\quad(i\in[r-3]),
\]
$|A_j|=O(f_j)$, and $H_j[A_j]$ admits a homomorphism to $KG(\frac{m_j}{2r-5}-O(f_j),O(f_j))$. Here the implicit constants are uniform in $j$, because the constants in the induction hypothesis depend only on $r$.

It remains to convert these estimates to the $n_j$-scale. We have
\[
m_j\geq\delta(G_j)\geq\frac{2r-5}{2r-3}n_j-f_j.
\]
The inductive partition contains the independent set $J_{1,j}$ of order at least $2m_j/(2r-5)-O(f_j)$. Every independent set $X$ in $G_j$ satisfies
\[
|X|\leq n_j-\delta(G_j)\leq\frac{2}{2r-3}n_j+f_j.
\]
Applying this to $J_{1,j}$ yields $m_j\leq\frac{2r-5}{2r-3}n_j+O(f_j)$. Consequently,
\[
m_j=\frac{2r-5}{2r-3}n_j+O(f_j),
\qquad
|I_j|=n_j-m_j=\frac{2}{2r-3}n_j+O(f_j),
\]
and hence
\[
\frac{m_j}{2r-5}=\frac{n_j}{2r-3}+O(f_j),
\qquad
\frac{2m_j}{2r-5}=\frac{2n_j}{2r-3}+O(f_j).
\]
Setting $J_{r-2,j}=I_j$ gives the asserted partition of $G_j$, including the required Kneser parameters, for every sufficiently large $j$. This contradicts the choice of the sequence and completes the induction.
\end{proof}

\section{Concluding remarks}\label{section-concluding remarks}

Theorem~\ref{thm-break the border} shows that, for $0<\delta<1$, every sufficiently large triangle-free graph with minimum degree at least $\frac{n}{3}-n^{1-\delta}$ has chromatic number at most
\[
10^{391}+1+\left\lceil\frac{(1+\varepsilon)(1-\delta)}{\delta}\right\rceil.
\]
Consequently, for every $t\geq10^{391}+3$ and $\varepsilon>0$, the function $s_t(n)$ in Problem~\ref{problem2} satisfies
\[
s_t(n)\geq n^{\,1-\frac{1+\varepsilon}{t-10^{391}-1+\varepsilon}}
\]
for all sufficiently large $n$. We have not optimized the constant factors. Our bound follows from the relation among odd girth, order, and chromatic number due to Kierstead, Szemer\'edi, and Trotter. It is natural to ask whether a matching upper bound for $s_t(n)$, up to constant factors, can be obtained by constructing an $n^{1-\delta}$-vertex subgraph of $KG(n,n^{1-\delta})$ with large chromatic number for which Theorem~\ref{thm-odd girth, chi, vertex number} is nearly tight. Alternatively, one may seek a stronger upper bound on the chromatic number of such Kneser subgraphs.

Although maximal triangle-free graphs with minimum degree greater than $\frac{n}{3}$ are known to be blow-ups of Vega graphs, the case of minimum degree exactly $\frac{n}{3}$ remains intriguing; this problem was also mentioned in \cite{2024-luczak}.
\begin{problem}
Characterize all the $n$-vertex maximal triangle-free graphs with minimum degree exactly $\frac{n}{3}$.
\end{problem}
We prove that all such graphs are $4$-colorable. As observed in \cite{2024-luczak}, there are many triangle-free regular graphs of degree $\frac{n}{3}$ that are not blow-ups of Vega graphs. This suggests the following question: must every exceptional graph be regular?
\begin{problem}[\cite{2024-luczak}]
Let $G$ be an $n$-vertex triangle-free graph with $\delta(G)\geq\frac{n}{3}$. If $G$ contains a vertex of degree greater than $\frac{n}{3}$, must $G$ be a blow-up of a Vega graph?
\end{problem}
The proof of Theorem~\ref{thm-main thm} may provide a starting point for a more detailed structural analysis.

\section*{Acknowledgement}
Jiaao Li and Xinyuan Li are partially supported by National Key Research and Development Program of China (No. 2022YFA1006400), National Natural Science Foundation of China (No. 12571371), Natural Science Foundation of Tianjin (No. 24JCJQJC00130), and the Fundamental Research Funds for the Central Universities, Nankai University. During the preparation of this paper, ChatGPT was used for proofreading and language editing to improve clarity and fluency. The mathematical ideas, proofs, and all technical mathematical content are the author's own work.



\end{document}